\documentclass[a4paper, 10pt]{article}
\usepackage[latin1]{inputenc}
\usepackage[english]{babel}

\usepackage{amsmath}
\usepackage{amsthm}
\usepackage{amssymb}
\usepackage{color}
\usepackage{bbm}
\usepackage{todonotes}

\theoremstyle{plain}
\newtheorem{theorem}{Theorem}
\newtheorem{proposition}[theorem]{Proposition}
\newtheorem{lemma}[theorem]{Lemma}
\newtheorem{corollary}[theorem]{Corollary}
\newtheorem{remark}[theorem]{Remark}

\newcommand{\R}{\mathbb{R}}

\newcommand{\mD}{\mathcal{D}}
\newcommand{\mL}{\mathcal{L}}

\newcommand{\mV}{\mathcal{V}}
\newcommand{\dd}{\, \text{d}}

\newcommand{\be}{\mathbf{e}}
\newcommand{\bff}{\mathbf{f}}
\newcommand{\bg}{\mathbf{g}}
\newcommand{\bG}{\mathbf{G}}
\newcommand{\bp}{\mathbf{p}}
\newcommand{\br}{\mathbf{r}}
\newcommand{\bu}{\mathbf{u}}
\newcommand{\bv}{\mathbf{v}}
\newcommand{\bw}{\mathbf{w}}

\newcommand{\by}{\mathbf{y}}
\newcommand{\bz}{\mathbf{z}}
\newcommand{\bphi}{\boldsymbol{\phi}}

\newcommand{\bbH}{\mathbb{H}}
\newcommand{\bbL}{\mathbb{L}}

\DeclareMathOperator{\divv}{div}

\DeclareMathOperator{\Symm}{Sym}

\newtheorem{innercustomthm}{Consequence}
\newenvironment{customthm}[1]
{\renewcommand\theinnercustomthm{#1}\innercustomthm}
{\endinnercustomthm}

\newtheorem{innercustomthmm}{Assumption}
\newenvironment{customthmm}[1]
{\renewcommand\theinnercustomthmm{#1}\innercustomthmm}
{\endinnercustomthmm}

\usepackage{hyperref}

\usepackage{vmargin}
\setmarginsrb{3cm}{2.5cm}{3cm}{1.5cm}{0cm}{0cm}{0cm}{1.5cm}

\begin{document}

\title{Feedback Stabilization of the Two-Dimensional Navier-Stokes Equations by Value Function Approximation}

\author{Tobias Breiten\footnote{Institute of Mathematics, University of Graz, Austria. E-mail: tobias.breiten@uni-graz.at} \quad
Karl Kunisch\footnote{Institute of Mathematics, University of Graz, Austria and RICAM Institute, Austrian Academy of Sciences, Linz, Austria. E-mail: karl.kunisch@uni-graz.at} \quad
Laurent Pfeiffer\footnote{Institute of Mathematics, University of Graz, Austria. E-mail: laurent.pfeiffer@uni-graz.at}
}

\maketitle

\begin{abstract}
The value function associated with an optimal control problem subject to the Navier-Stokes equations in dimension two is analyzed.
Its smoothness is established around a steady state, moreover, its derivatives are shown to satisfy a Riccati equation at the order two and generalized Lyapunov equations at the higher orders. An approximation of the optimal feedback law is then derived from the Taylor expansion of the value function. A convergence rate for the resulting controls and closed-loop systems is demonstrated.
\end{abstract}

{\em Keywords:} Stabilization, 2-D Navier-Stokes equations, value function, Taylor expansion, feedback control.

{\em AMS Classification:}
35Q35, 49J20, 49N35, 93D05, 93D15.

\section{Introduction}

In this work we continue our investigations of the value function associated with infinite-horizon optimal control problems of partial differential equations, that we initiated in \cite{BreKP18,BreKP19}. We consider
a stabilization problem of the Navier-Stokes equations in dimension two and focus on the regularity of the value function and its characterization as a solution to a Hamilton-Jacobi-Bellman (HJB) equation.
This task has been the subject of tremendous research, for optimal control problems of a general structure, in general associated with finite-dimensional dynamical systems. The use of the notion of viscosity solutions has allowed to deal with the low regularity of the value function.
In the present paper, to the contrary, we show that the value function is smooth and that the HJB equation is satisfied in the strict sense, in a neighborhood of the steady state.
Moreover, we show that the derivatives of the value function, at the steady state, are solutions to an algebraic Riccati equation (for the order 2) and to linear equations, called generalized Lyapunov equations, for the higher orders. The main interest of these results is the fact that polynomial feedback laws can be derived from Taylor approximations of the value function. Moreover their efficiency can be analyzed.

From a methodological point of view, we mainly follow the techniques that we laid out for bilinear optimal control problems (such as control problems of the Fokker-Planck equation) in \cite{BreKP19} and \cite{BreKP18}.
The Navier-Stokes control system considered here requires a different functional analytic treatment.
In fact, the involved nonlinear terms must be tackled with different estimates, to guarantee, for example, the well-posedness of the closed-loop system. They also lead to different generalized Lyapunov equations.
Moreover, from the point of view of open-loop control of the Navier-Stokes equation, this paper contains results on infinite-horizon optimal control which are not readily available elsewhere.

Feedback stabilization of the Navier-Stokes equations has been and still is an active topic of research. Among the numerous works, we refer to, e.g., \cite{Bar11,BarLT06,BenH16, Fur01,Ray06}, and the references therein. For literature concerning open-loop optimal control of the Navier-Stokes equations, we can only cite a small selection \cite{BTZ2000,C98,CGHUU02,DG08,DI94,GHS91,HinK01,TW06}.

The technique of approximation of the value function with a Taylor expansion dates back to \cite{Alb61,Luk69}, where optimal control problems associated to finite-dimensional control systems were investigated. We also quote follow-up work, for instance in \cite{AguK14,BeeTB00,NavK07}. For infinite-dimensional problems, we are only aware of \cite{BreKP18,BreKP19}. In \cite{BreKP17c}, the numerical solvability of the Lyapunov equations has been addressed. Model reduction techniques based on balanced truncation have been used in this reference to cope with the curse of dimensionality encountered when dealing with PDE controlled systems.

Let us next specify the problem which will be investigated in this paper. Throughout $\Omega \subset \mathbb R^2$ denotes a bounded domain with Lipschitz boundary $\Gamma$. Given two vector valued functions $\boldsymbol{\varphi}$ and $\boldsymbol{\psi}$, we consider a solution $(\bar{\bz},\bar{q})$ of the stationary Navier-Stokes equations
\begin{equation}\label{eq:stat_NSE}
 \begin{aligned}
  -\nu \Delta \bar{\bz} + (\bar{\bz} \cdot \nabla ) \bar{\bz} + \nabla \bar{q}  &= \boldsymbol{\varphi} \quad \text{in } \Omega, \\
  \divv\bar{\bz} &= 0 \quad \text{in } \Omega, \\
  \bar{\bz} &=\boldsymbol{\psi} \quad \text{on } \Gamma.
 \end{aligned}
\end{equation}
Our goal is to find a control $u$ such that the solution $(\bz,q)$ to the transient Navier-Stokes equations
\begin{equation}\label{eq:inhom_NSE}
  \begin{aligned}
   \frac{\partial \bz}{\partial t} &= \nu \Delta \bz - (\bz \cdot \nabla)\bz - \nabla q +\boldsymbol{\varphi} + \tilde{B} u &&  \text{in } \Omega \times (0,T), \\
   \divv\bz&= 0 &&  \text{in } \Omega \times (0,T), \\
   \bz &= \boldsymbol{\psi} &&  \text{on } \Gamma \times (0,T), \\
   \bz(0) &= \bar{\bz} + \by_0 &&
  \end{aligned}
\end{equation}
is stabilized around $\bar{\bz},$ i.e., $\lim\limits_{t\to \infty} \bz(t) = \bar{\bz}$ provided the initial perturbation $\by_0$ is \emph{small} in an appropriate sense. The control operator $\tilde{B}$ will be defined below. Throughout this work, we assume that $\divv \by_0=0$. Our results are concerned with feedback stabilization of \eqref{eq:inhom_NSE} and for this purpose, we consider new state variables $(\by,p):=(\bz,q)-(\bar{\bz},\bar{q})$ which satisfy the following generalized Navier-Stokes equations
\begin{equation}\label{eq:gen_hom_NSE}
  \begin{aligned}
   \frac{\partial \by}{\partial t} &= \nu \Delta \by - (\by \cdot \nabla)\bar{\bz} - (\bar{\bz} \cdot \nabla)\by-(\by \cdot \nabla)\by - \nabla p + \tilde{B} u &&  \text{in } \Omega \times (0,T), \\
   \divv \by &= 0 &&  \text{in } \Omega \times (0,T), \\
   \by &= 0 &&  \text{on } \Gamma \times (0,T), \\
   \by(0) &= \by_0.&&
  \end{aligned}
\end{equation}

The following sections are structured as follows. The  problem statement and fundamental results on the state-equation on the time interval $[0,\infty)$ are given in Section 2. Section 3  contains the existence theory of optimal controls, the adjoint equation, sensitivity analysis, and differentiability of the value function. The characterization of all higher order derivatives of the value as solutions to generalized Lyapunov equations are provided in Section 4. Section 5 contains the Taylor expansion of the value function, and  estimates for convergence rates between the optimal solution and its approximation on the basis of feedback solutions obtained from derivatives of the value function. The paper closes with a very short outlook.

\paragraph{Notation.}
For Hilbert spaces $V\subset Y$ with dense and compact embedding, we consider the Gelfand triple $V\subset Y \subset  V'$ where $V'$ denotes the topological dual of $V$ with respect to the pivot space $Y$. Given $T\in \mathbb R$ we consider the space
\begin{align*}
W(0,T) =\left\{ y \in L^2(0,T;V) \ | \ \frac{\mathrm{d}}{\mathrm{d}t} y \in L ^2(0,T;V') \right\}.
\end{align*}
For $T=\infty$, the space $W(0,T)$ will be denoted by $W_\infty$. For vector-valued functions $\bff \in (L^2(\Omega))^2$, we use the notation $\bff\in \mathbb{L}^2(\Omega)$. Elements $\bff \in \mathbb L^{2}(\Omega)$ will be denoted in boldface and are distinguished from real-valued functions $g \in L^2(\Omega)$. Similarly, we use $\mathbb{H}^2(\Omega)$ for the space $(H^2(\Omega))^2$ and $\mathbb{H}^1_0(\Omega)$ for $(H^1_0(\Omega))^2$. Given a closed, densely defined linear operator $(A,\mD(A))$ in $Y$, its adjoint (again considered as an operator in $Y$) will be denoted with $(A^*,\mD(A^*))$.

Let us introduce some notation that will be needed for the description of polynomial mappings.
For $\delta \geq 0$ and a Hilbert space $Y$, we denote by $B_Y(\delta)$ the closed ball in $Y$ with radius $\delta$ and center 0. For $k \geq 1$, we make use of the following norm:
\begin{equation} \label{eq:NormYk}
\| (y_1,\dots,y_k) \|_{Y^k}= \max_{i=1,\dots,k} \| y_i \|_Y.
\end{equation}
Given a Hilbert space $Z$, we say that $\mathcal{T}\colon Y^k \rightarrow Z$ is a bounded multilinear mapping (or bounded multilinear form when $Z= \R$) if for all $i \in \{ 1,\dots,k \}$ and for all $(z_1,\dots,z_{i-1},z_{i+1},\dots,z_k) \in Y^{k-1}$, the mapping
$z
\in Y \mapsto \mathcal{T}(z_1,\dots,z_{i-1},z,z_{i+1},\dots,z_k) \in Z$ is linear and
\begin{equation} \label{eq:OperatorNormTensor}
\| \mathcal{T} \| := \sup_{y \in B_{Y^k}(1)} \| \mathcal{T}(y) \|_Z < \infty.
\end{equation}
The set of bounded multilinear mappings on $Y^k$ will be denoted by $\mathcal{M}(Y^k,Z)$. For all $\mathcal{T} \in \mathcal{M}(Y^k,Z)$ and for all $(z_1,\dots,z_k) \in Y^k$,
\begin{equation*}
\| \mathcal{T}(z_1,\dots,z_k) \|_Z \leq \| \mathcal{T} \| \, \prod_{i=1}^k \| z_i \|_Y.
\end{equation*}
Given a bounded multilinear form $\mathcal{T}$ and $z_2,\dots,z_k \in Y^{k-1}$, we denote by $\mathcal{T}(\cdot,z_2,\dots,z_k)$ the bounded linear form $z_1 \in Y \mapsto \mathcal{T}(z_1,\dots,z_k) \in \R$. It will be very often identified with its Riesz representative. Note that
\begin{equation} \label{eq:Riesz_identification}
\| \mathcal{T}(\cdot,z_2,\dots,z_k) \|_Y
= \sup_{z_1 \in B_Y(1)} \mathcal{T}(z_1,\dots,z_k)
\leq \| \mathcal{T} \| \prod_{i=2}^k \| z_i \|_Y.
\end{equation}
Bounded multilinear mappings $\mathcal{T} \in \mathcal{M}(Y^k,Z)$ are said to be symmetric if for all $z_1,\dots,z_k \in Y^k$ and for all permutations $\sigma$ of $\{ 1,\dots,k \}$,
\begin{equation*}
\mathcal{T}(z_{\sigma(1)},\dots,z_{\sigma(k)})= \mathcal{T}(z_1,\dots,z_k).
\end{equation*}
Finally, given two multilinears mappings $\mathcal{T}_1 \in \mathcal{M}(Y^k,Z)$ and $\mathcal{T}_2 \in \mathcal{M}(Y^{\ell},Z)$, we denote by $\mathcal{T}_1 \otimes \mathcal{T}_2$ the bounded multilinear form defined by
\begin{equation*}
\mathcal{T}_1 \otimes \mathcal{T}_2(z_1,\dots,z_{k+ \ell})
= \langle \mathcal{T}_1(z_1,\dots,z_k), \mathcal{T}_2(z_{k+1},\dots,z_{k+\ell}) \rangle_Z.
\end{equation*}

Throughout the manuscript, we use $M$ as a generic constant that might change its value between consecutive lines.

\section{Problem formulation}

\subsection{Abstract Cauchy problem}

In this section, we formulate system \eqref{eq:gen_hom_NSE} as an abstract Cauchy problem on a suitable Hilbert space and, subsequently, define the stabilization problem of interest. This procedure is quite standard, see, for instance, \cite{Bar11,BarLT06,Fur01,Ray06,Tem79} for details. We introduce the spaces
\begin{align*}
  Y&:=\left\{ \by \in \bbL^2(\Omega)\, |\, \divv \by =0,\, \by\cdot \vec{n}=0 \text{ on } \Gamma \right\}, \\
  V&:=\left\{ \by \in \bbH_0^1(\Omega)\, | \, \divv \by =0 \right\}.
\end{align*}
It is well-known that $Y$ is a closed subspace of $\bbL^2(\Omega)$. Moreover, we have the orthogonal decomposition
\begin{align}\label{eq:orth_dec_div_free}
  \bbL^2(\Omega) = Y \oplus Y^\perp,
\end{align}
where
\begin{align}\label{eq:Yperp}
  Y^\perp = \left\{\bz =\nabla p \, | \, p \in H^1(\Omega) \right\},
\end{align}
see, e.g., \cite[page 15]{Tem79}. By $P$ we denote the \textit{Leray projector} $P\colon \bbL^2(\Omega) \to Y$ which is the orthogonal projector in $\bbL^{2}(\Omega)$ onto $Y$. Following, e.g., \cite{Bar11}, we define a trilinear form $s$ by
\begin{align}\label{eq:trilinear_form}
 s(\bu,\bv,\bw):=\int_{\Omega} \sum_{i,j=1}^2 u_i w_j \frac{\partial v_j}{\partial x_i}   \dd x = \langle (\bu \cdot \nabla)\bv,\bw \rangle_{\bbL^{2}(\Omega)} , \ \ \forall \bu,\bv,\bw \in V
\end{align}
and a nonlinear operator $F\colon V\to V'$ by
\begin{align}\label{eq:nonlinearity}
 \langle F(\by),\bw \rangle_{V',V} := s(\by,\by,\bw), \ \ \forall \bw \in V.
\end{align}
For the bilinear mapping associated with the linearization of $F$, we introduce the operator
\begin{align}\label{eq:def_bilinear_form}
N\colon V\times V \to V', \ \  \langle N(\by,\bz),\bw \rangle_{V',V}:=s(\by,\bz,\bw).
\end{align}
The \emph{Oseen-Operator} is then defined by
\begin{align}\label{eq:def_oseen}
	A_{0} \colon V\times V \to V', \ \ \langle A_{0}(\by,\bz),\bw \rangle_{V',V} := \langle N(\by,\bz)+N(\bz,\by),\bw \rangle_{V',V}.
\end{align}
The following well-known results (see, e.g., \cite{Bar11}, \cite[Lemma III.3.4]{Tem79}) concerning $s$ and $N$ will be used frequently throughout the paper.

\begin{proposition}\label{prop:estimates_oseen}
The following properties hold for $N$ and $s$:
\begin{itemize}
\item[(i)] $\|N(\by,\bz)\|_{V'}\le M \|\by\|^{\frac{1}{2}}_{Y} \| \bz\|^{\frac{1}{2}}_{Y} \|\by\|^{\frac{1}{2}}_{V} \| \bz\|^{\frac{1}{2}}_{V} $, \text{for all} $\by, \bz \in V$,
  \item[(ii)] $s(\by,\bz,\bw)=-s(\by,\bw,\bz)$, \text{for all} $\by,\bz,\bw \in V$.
\end{itemize}
\end{proposition}

With the previous result, we obtain similar properties for time-varying functions $\by,\bz,\bw$.

\begin{lemma}\label{lem:estimates_oseen_time_var}
Let $T \in (0,\infty]$. For all $\by \in W(0,T)$, for all $\bz \in W(0,T)$, and for all $\bw \in L^2(0,T;V)$,
\begin{align*}
& \langle N(\by,\bz),\bw \rangle_{L^2(0,T;V'),L^2(0,T;V)} \\
& \qquad \leq M \| \by \|_{L^\infty(0,T;Y)}^{\frac{1}{2}}
\| \by \|_{L^2(0,T;V)}^{\frac{1}{2}}
\| \bz \|_{L^\infty(0,T;Y)}^{\frac{1}{2}}
\| \bz \|_{L^2(0,T;V)}^{\frac{1}{2}}
\| \bw \|_{L^2(0,T;V)}.
\end{align*}
Moreover, if $\bw \in L^\infty(0,T;V)$,
\begin{align*}
& \langle N(\by,\bz),\bw \rangle_{L^2(0,T;V'),L^2(0,T;V)} \\
& \qquad \leq M \| \by \|_{L^2(0,T;Y)}^{\frac{1}{2}}
\| \by \|_{L^2(0,T;V)}^{\frac{1}{2}}
\| \bz \|_{L^2(0,T;Y)}^{\frac{1}{2}}
\| \bz \|_{L^2(0,T;V)}^{\frac{1}{2}}
\| \bw \|_{L^\infty(0,T;V)},
\end{align*}
where $M$ is the constant given by Proposition \ref{prop:estimates_oseen}.

\end{lemma}

\begin{proof}
Using Proposition \ref{prop:estimates_oseen} and Cauchy-Schwarz inequality (two times), we obtain that
\begin{align*}
&\langle N(\by,\bz),\bw \rangle_{L^{2}(0,T;V'),L^{2}(0,T;V)} \leq  \ M \int_0^T \| \by(t) \|_Y^{\frac{1}{2}} \| \by(t) \|_V^{\frac{1}{2}} \| \bz(t) \|_Y^{\frac{1}{2}} \| \bz(t) \|_V^{\frac{1}{2}} \| \bw(t) \|_V \mathrm{d}t \\
 & \quad \leq\ M \| \by \|_{L^2(0,T;V)}^{\frac{1}{2}} \| \bz \|_{L^2(0,T;V)}^{\frac{1}{2}}
\Big( \int_0^T \| \by(t) \|_Y \| \bz(t) \|_Y \| \bw(t) \|_V^2 \Big)^{\frac{1}{2}}.
\end{align*}
The two inequalities easily follow.
\end{proof}

\begin{corollary}\label{cor:estimates_oseen_time_var}
There exists $M>0$ such that for all $\by$ and $\bz \in W_\infty$,
\begin{equation*}
\| N(\by,\bz) \|_{L^2(0,\infty;V')} \leq M \| \by \|_{W_\infty} \| \bz \|_{W_\infty}.
\end{equation*}
\end{corollary}

For $\bar{\bz} \in V$,  we further introduce the \textit{Stokes-Oseen} operator $A$ via
\begin{equation}\label{eq:A_NSE}
  \mD(A) = \bbH^2(\Omega)\cap V, \ \ A\by = P(\nu \Delta \by - (\by \cdot \nabla)\bar{\bz} - (\bar{\bz} \cdot \nabla)\by ).
\end{equation}
Considered as operator in $\bbL^2(\Omega)$ the adjoint $A^*$,  as operator in $\bbL^2(\Omega)$, can be characterized by (see, e.g., \cite{Ray06})
\begin{equation}\label{eq:A_adjoint_NSE}
  \mD(A^*) = \bbH^2(\Omega)\cap V, \ \ A^*\bp = P(\nu \Delta \bp - (\nabla \bar{\bz})^T \bp + (\bar{\bz} \cdot \nabla)\bp ).
\end{equation}
We note that as a consequence of Proposition \ref{prop:estimates_oseen}, the operator $A$ can be extended to a bounded linear operator from $V$ to $V'$ in the following manner:
\begin{equation*}
\langle A \by , \bw \rangle_{V',V}
= - \nu \langle \nabla \by, \nabla \bw \rangle_{\bbL^2(\Omega)}
- \langle A_0(\bar{\bz}, \by), \bw \rangle_{V',V}.
\end{equation*}
Note that this extension is consistent, since by definition of the Leray projector $P$, we have $\langle P \by, \bw \rangle_Y = \langle \by, \bw \rangle_Y$ for all $\by \in \bbL^2(\Omega)$ and for all $\bw \in V$. Similarly, $A^*$ can be extended to a bounded linear operator from $V$ to $V'$.

The control operator is chosen to satisfy $\tilde B\in \mathcal{L}(U,\bbL^2(\Omega))$. We further define $B:=P\tilde{B} \in \mathcal{L}(U,Y)$. The controlled state equation \eqref{eq:gen_hom_NSE} can now be formulated as the abstract control system
\begin{equation}\label{eq:abs_Cauchy_NSE}
\begin{aligned}
  \dot{\by}(t) &= A\by - F(\by) + Bu, \quad \by(0) =\by_0,
\end{aligned}
\end{equation}
where the pressure $p$ is eliminated.
We can finally formulate the stabilization problem as an infinite-horizon optimal control problem:
\begin{equation} \label{eq:NLQprob} \tag{$P$}
\inf_{\begin{subarray}{c} \by \in W_\infty \\ u \in L^2(0,\infty;U) \end{subarray}} J(\by,u), \quad \text{subject to: } e(\by,u)= (0,\by_0)
\end{equation}
where $J \colon W_\infty \times L^2(0,\infty;U) \rightarrow \R$ and $e \colon W_\infty \times L^2(0,\infty;U) \rightarrow L^2(0,\infty;V') \times Y$ are defined by
\begin{align}
J(\by,u)= & \frac{1}{2} \int_0^\infty \| \by \|^2_Y \dd t + \frac{\alpha}{2} \int_0^\infty \| u(t) \|_U^2 \dd t \\
e(\by,u)= & \big( \dot{\by}-(A\by - F(\by) + Bu), \by(0) \big).
\end{align}
Let us note that $e\colon W_\infty\times L^2(0,\infty;U) \to L^2(0,\infty;V') \times Y$
is well-defined by Corollary \ref{cor:estimates_oseen_time_var}.

\subsection{Assumptions and first properties}

Throughout the article we assume that the following assumptions hold true.
\begin{customthmm}{A1}\label{ass:A1}
 The stationary solution satisfies $\bar{\bz} \in V$.
 \end{customthmm}
 \begin{customthmm}{A2}\label{ass:A2}
There exists an operator $K \in \mathcal{L}(Y,U)$ such that the semigroup $e^{(A-BK) t}$ is exponentially stable on $Y$.
\end{customthmm}
Assumption \ref{ass:A2} concerning the exponential feedback stabilizability of the Stokes-Oseen
operator  is well investigated. We refer e.g. to \cite{Bar11} where finite-dimensional
feedback operators are constructed on the basis of spectral decomposition or alternatively
by Riccati theory. In this case \ref{ass:A2} can be satisfied with $U=\mathbb{R}^m$, for $m$ appropriately large. Alternatively, we can rely on exact controllability results as obtained  in
\cite{FGIP2011}. They imply that  the finite cost criterion holds. We can then rely on classical results,
see, e.g., \cite{PZ1981} which guarantee the existence of a stabilizing feedback operator.

Let us discuss some important consequences of the above definitions and assumptions.

\begin{customthm}{C1}\label{cons:C1}
There exists $\lambda \geq 0$ and $\theta >0$ such that
\begin{align}\label{eq:VY-coercivity}
\langle (\lambda I-A)\bv,\bv \rangle_Y \ge \theta \|\bv \|_V^2, \quad \text{for all } \bv \in V.
\end{align}
Hence, $A$ generates an analytic semigroup $e^{At}$ on $Y$, see \cite[Part II, Chapter 1, Theorem 2.12]{Benetal07}.
\end{customthm}

\begin{customthm}{C2}\label{cons:C2}
For all $\by_0 \in Y$, for all $\bff \in L^2(0,\infty;V')$, and for all $T>0$, there exists a unique solution $\by \in W(0,T)$ to the system
\begin{align*}
 \dot{\by}=A\by + \bff , \quad \by(0)=\by_0.
\end{align*}
This solution satisfies
\begin{align*}
 \| \by \|_{W(0,T)} \le c(T) (\| \by _0\|_Y + \| \bff \| _{L^2(0,\infty;V')} )
\end{align*}
with a continuous function $c$. Assuming that  $\by \in L^2(0,\infty;Y)$, we consider the equivalent equation
\begin{align*}
 \dot{\by}=\underbrace{(A- \lambda I)}_{A_\lambda} \by + \underbrace{\lambda \by + \bff }_{\bff_\lambda}, \quad \by (0)=\by_0,
\end{align*}
where $\bff_\lambda\in L^2(0,\infty;V')$. By \eqref{eq:VY-coercivity}, the operator  $A_\lambda$ generates an analytic, exponentially stable, semigroup on $Y$ satisfying $\| e^{A_\lambda t} \|_Y \le  e^{- \delta t}$ for some $ \delta > 0$ independent of $t \geq 0$, see \cite[Theorem II.1.2.12]{Benetal07}. It follows that $\by \in W_\infty$ and there exists $M_\lambda$ such that with
\begin{align}\label{eq:reg_est_shift}
\| \by \| _{W_\infty} \le M_{\lambda} ( \| \by _0 \| _Y + \| \bff_{\lambda} \| _{L^2(0,\infty;V')} ).
\end{align}
This estimate is obtained by adapting \cite[Corollary II.3.2.1]{Benetal07} and \cite[Theorem II.3.2.2]{Benetal07} from the temporal domain $(0,T)$ to $(0,\infty)$, which can be achieved using the exponential stability of $e^{A_\lambda t}$.
\end{customthm}

\begin{lemma} \label{lemma:Lipschitz_F}
 There exists a constant $C>0$ such that for all $\delta \in [0,1]$ and for all $\by$ and $\bz \in W_\infty$ with $\|\by\|_{W_\infty}\le \delta$ and $\|\bz\|_{W_\infty}\le \delta$, it holds that
 \begin{align*}
   \|F(\by)-F(\bz)\|_{L^2(0,\infty;V')} \le \delta C\| \by-\bz\|_{W_\infty}.
 \end{align*}
\end{lemma}

\begin{proof}
We have
 \begin{align*}
 \| F(\by)-F(\bz)\|_{L^2(0,\infty;V')} &= \| N(\by,\by)-N(\bz,\bz)\|_{L^2(0,\infty;V')} \\
 &\le \|N(\by-\bz,\by)\|_{L^2(0,\infty;V')} + \| N(\bz,\by-\bz)\|_{L^2(0,\infty;V')}.
 \end{align*}
 The assertion now easily follows  from Corollary \ref{cor:estimates_oseen_time_var}.
\end{proof}

The following lemma is formulated for an abstract generator $A_{s}$ of an analytic semigroup on $Y$. It will subsequently be used to address the asymptotic behavior of the nonlinear system \eqref{eq:abs_Cauchy_NSE}. We point out that the statement is similar to \cite[Theorem 6.1]{Ray06} which, since it addresses the boundary control case, assumes a slightly more regular initial condtion $\by_0\in \bbH^\varepsilon(\Omega)\cap Y$.

\begin{lemma} \label{lem:eq:non_loc_sol}
  Let $A_s$ be the generator of an exponentially stable analytic semigroup $e^{A_s t}$ on $Y$ such that \eqref{eq:VY-coercivity} holds. Let $C$ denote the constant from Lemma \ref{lemma:Lipschitz_F}. Then there exists a constant $M_s$ such that for all $\by_0 \in Y$ and $\bff \in L^2(0,\infty;V')$ with
  \begin{align*}
\gamma:=    \| \by _0\| _Y + \| \bff \| _{L^2(0,\infty;V')} \le \frac{1}{4CM_s^2}
  \end{align*}
  the system
  \begin{equation}\label{eq:loc_stab_non}
\dot{\by}= A_s\by - F(\by) +\bff, \quad \by(0)= \by_0
\end{equation}
has a unique solution $\by$ in $W_\infty$, which moreover satisfies
\begin{equation*}
\|\by\|_{W_\infty} \le 2 M_s \gamma.
\end{equation*}
\end{lemma}

\begin{proof}
We follow the line of argumentation provided in the proof \cite[Theorem 6.1]{Ray06}.
 Since the semigroup $e^{A_s t}$ is exponentially stable on $Y$, it follows that for all $(\by_0,\bg)\in Y\times L^2(0,\infty;V')$ the system
 \begin{align*}
  \dot{\bz}=A_s \bz + \bg , \quad \bz(0)=\by_0
 \end{align*}
has a unique solution $\bz \in W_\infty$. Moreover, there exists a constant $M_s$ such that
\begin{align}\label{eq:reg_est_stabsys}
  \| \bz \|_{W_\infty} \le M_s (\| \by _0 \| _Y + \| \bg \|_{L^2(0,\infty;V')} ).
\end{align}
Without loss of generality we can assume that $M_s\ge \frac{1}{2C}$. We claim that the constant $M_s$ is the one announced in the assertion. This will be shown by a fixed-point argument applied to the system \eqref{eq:loc_stab_non}. For this purpose, let us define $\mathcal{M}=\left\{ \by \in W_\infty \ | \ \|\by\|_{W_\infty} \le 2 M_s \gamma \right\}$
and let us define the mapping $\mathcal{Z}\colon \mathcal{M} \ni \by \mapsto \bz=\mathcal{Z}(\by)\in W_\infty$, where $\bz$ is the unique solution of
\begin{align*}
\dot{\bz}=A_s\bz - F(\by) + \bff , \quad \bz(0)=\by_0.
\end{align*}
If there exists a unique fixed point of $\mathcal{Z}$, then it is a unique solution of \eqref{eq:loc_stab_non} in $\mathcal{M}$. With $C$ and $M_s$ given, we shall use Lemma
\ref{lemma:Lipschitz_F} with $\delta= 2M_s\gamma\le\frac{1}{2CM_s}\le 1$.
Together with \eqref{eq:reg_est_stabsys}, it follows that
\begin{align*}
  \| \bz \| _{W_\infty} &\le M_s ( \| F(\by) \|_{L^2(0,\infty;V')} + \| \bff \|_{L^2(0,\infty;V')} + \| \by_0 \| _Y )  \\
 & \le M_s \left(\frac{1}{2M_s} \| \by \|_{W_\infty}+ \gamma \right) \le 2  M_s  \gamma.
\end{align*}
This implies $\mathcal{Z}(\mathcal{M})\subseteq \mathcal{M}$. For $\by_1,\by_2\in \mathcal{M}$ consider now $\bz=\mathcal{Z}(\by_1)-\mathcal{Z}(\by_2)$ solving
\begin{align*}
  \dot{\bz}= A_s\bz - F(\by_1)+F(\by_2), \quad \bz(0)=0.
\end{align*}
Again by \eqref{eq:reg_est_stabsys} and Lemma \ref{lemma:Lipschitz_F} we obtain
\begin{align*}
\| \mathcal{Z}(\by_1)-\mathcal{Z}(\by_2) \| _{W_\infty} &=  \| \bz \| _{W_\infty}\le M_s (\| F(\by_1)-F(\by_2) \|_{L^2(0,\infty;V')} ) \\
 &\le M_s \delta C \| \by_1 - \by_2 \| _{W_\infty} \le \frac{1}{2} \| \by_1 -\by_2 \| _{W_\infty}.
\end{align*}
In other words, $\mathcal{Z}$ is a contraction in $\mathcal{M}$ and therefore, there exists a unique $\by \in \mathcal{M}$ such that $\mathcal{Z}(\by)=\by$. Regarding uniqueness in $W_{\infty}$, consider two solutions $\by,\bz \in W_{\infty}$. For the difference $\be :=\by-\bz$ it then holds
\begin{align*}
	\dot{\be}=A_{s} \be  - F(\by)+F(\bz), \ \ \be(0)=0.
\end{align*}
Multiplying with $\be$ and taking inner products yields
\begin{align*}
	\frac{1}{2}\frac{\mathrm{d}}{\mathrm{d}t}\| \be \|_{Y}^{2} = \langle A_{s} \be ,\be \rangle_{Y} - \langle F(\by)-F(\bz),\be \rangle_{V',V}.
\end{align*}
Since $A_{s}$ satisfies an inequality of the form \eqref{eq:VY-coercivity}, we have
\begin{align*}
	\frac{1}{2}\frac{\mathrm{d}}{\mathrm{d}t} \| \be \| _{Y}^{2} \le \alpha \| \be \|_{Y}^{2} - \beta \| \be \|_{V}^{2} + \| F(\by)-F(\bz)\|_{V'} \| \be \|_{V},
\end{align*}
where $\alpha \ge 0$ and $\beta >0 $.
Using Proposition \ref{prop:estimates_oseen} and Young's inequality we further obtain
\begin{align*}
	\frac{1}{2}\frac{\mathrm{d}}{\mathrm{d}t} \| \be \| _{Y}^{2} &\le \alpha \| \be \|_{Y}^{2} - \beta \| \be \|_{V}^{2} + M\left( \| \be \|_{Y}^{\frac{1}{2}}\| \by \|_{Y}^{\frac{1}{2}}\| \be \|_{V}^{\frac{1}{2}} \| \by \|_{V}^{\frac{1}{2}}+ \| \be \|_{Y}^{\frac{1}{2}}\| \bz \|_{Y}^{\frac{1}{2}}\| \be \|_{V}^{\frac{1}{2}} \| \bz \|_{V}^{\frac{1}{2}} \right) \| \be \|_{V}\\
	&\le \alpha \| \be \|_{Y}^{2} - \beta \| \be \|_{V}^{2} + \frac{M}{ \iota}\| \be \|^{2}_{V}  + \frac{M\iota}{2}\| \be \| _{V} \left(  \| \be \|_{Y}\| \by \|_{Y} \| \by \|_{V}+ \| \be \|_{Y} \| \bz \|_{Y} \| \bz \|_{V}\right) \\
	&\le \alpha \| \be \|_{Y}^{2} - \beta \| \be \|_{V}^{2} + \frac{M}{ \iota}\| \be \|_{V} ^{2} + \frac{M\iota}{2\kappa} \|\be\|_{V}^{2}+ \frac{M\iota\kappa}{4} \left(  \| \be \|_{Y}^{2}\| \by \|_{Y} ^{2}\| \by \|_{V}^{2}+ \| \be \|_{Y}^{2}\| \bz \|_{Y}^{2} \| \bz \|_{V}^{2}\right) .
\end{align*}
Taking $\iota$ and $\kappa$ sufficiently large, it holds that
\begin{align*}
	\frac{1}{2}\frac{\mathrm{d}}{\mathrm{d}t} \| \be \| _{Y}^{2} &\le \left(\alpha + \frac{M\iota\kappa}{4} \left(  \| \by \|_{Y} ^{2}\| \by \|_{V}^{2}+ \| \bz \|_{Y}^{2} \| \bz \|_{V}^{2}\right) \right) \|\be \|_{Y}^{2}.
\end{align*}
Since $\by,\bz \in W_{\infty}$ and $\be(0)=0$, with Gronwall's inequality, we conclude that $\be(t)=0$ for all $t\ge 0$. Hence, $\by=\bz$ showing the uniqueness of the solution in $W_{\infty}$.
\end{proof}

The following two corollaries are consequences of Lemma \ref{lemma:Lipschitz_F} and Lemma \ref{lem:eq:non_loc_sol}. The constant $C$ which is employed is the one given by Lemma \ref{lemma:Lipschitz_F}.

\begin{corollary}\label{cor:feas_control}
There exists a constant $M_K>0$ such that for all $\by_0\in Y$ and for all $\bff\in L^2(0,\infty;V')$ with
\begin{align*}
\gamma:=\| \by_0 \|_Y + \| \bff \| _{L^2(0,\infty;V')} \le \frac{1}{4CM_K^2} 
\end{align*}
there exists a control $u \in L^2(0,\infty;U)$ such that the system
\begin{equation}\label{eq:loc_stab_non2}
\dot{\by}= A\by + Bu - F(\by) +\bff, \quad \by(0)= \by_0
\end{equation}
has a unique solution $\by \in W_\infty$ satisfying
\begin{align*}
\|\by \|_{W_\infty} \le 2M_K \gamma  \quad \text{and} \quad \|u\|_{L^2(0,\infty;U)}\le 2 \|K\|_{\mathcal{L}(Y)} M_K \gamma.
\end{align*}
\end{corollary}

\begin{proof}
By assumption \ref{ass:A2}, there exists $K$ such that $A-BK$ generates an exponentially stable, analytic  semigroup on $Y$. The result then follows by applying Lemma \ref{lem:eq:non_loc_sol} to the system
\begin{align*}
   \dot{\by}=(A-BK)\by -F(\by)+\bff , \quad \by(0)=\by_0.
  \end{align*}
and by defining $u= -K\by$.
\end{proof}

In the following corollary, we assume without loss of generality that the constant $M_\lambda$ given by Consequence \ref{cons:C2} is such that $M_\lambda \geq \frac{1}{2C}$.

\begin{corollary}\label{cor:more_regularity}
Let $(\by_0,\bff)\in Y\times L^2(0,\infty;V')$ let $u \in L^2(0,\infty;U)$ be such that the system
  \begin{align*}
    \dot{\by} = A\by - F(\by) + Bu + \bff, \ \ \by(0)=\by_0
  \end{align*}
has a solution $\by \in L^2(0,\infty;Y)$.  If
  \begin{align*}
\gamma :=\| \by_0 \|_Y + \| \bff + \lambda \by + Bu \| _{L^2(0,\infty;V')} \le \frac{1}{4CM_\lambda^2} ,
\end{align*}
then $\by \in W_\infty$ and it holds that
\begin{align*}
  \| \by \| _{W_\infty } \le 2M_{\lambda} \gamma.
  \end{align*}
\end{corollary}

\begin{proof}
Since $\by \in L^2(0,\infty;Y)$, we can apply Lemma \ref{lem:eq:non_loc_sol} to the equivalent system
  \begin{align*}
   \dot{\by} = (A-\lambda I)\by- F(\by) +\tilde{\bff},
  \end{align*}
where $\tilde{\bff} = \bff + \lambda \by + Bu$. This shows the assertion.
\end{proof}

\section{Differentiability of the value function}

In this section we perform a sensitivity analysis for the stabilization problem.
The main purpose is to analyze the dependence of solutions to \eqref{eq:NLQprob} with respect to the initial condition $\by_{0}$ and to show the differentiability of the associated \emph{value function}, defined by
\begin{align*}
\mV(\by_{0}) = \inf_{\begin{subarray}{c} \by \in W_\infty \\ u \in L^2(0,\infty;U) \end{subarray}} J(\by,u), \quad \text{subject to: } e(\by,u)= (0,\by_0).
\end{align*}

\subsection{Existence of a solution and optimality conditions}

In Lemma \ref{lem:ex_opt_sol} we prove the existence of a solution $(\bar{\by},u)$ to problem \eqref{eq:NLQprob}, assuming that $\| \by_0 \|_Y$ is sufficiently small. We derive then in Proposition \ref{proposition:optiCondWeak} first-order necessary optimality conditions.

\begin{lemma}\label{lem:ex_opt_sol}
There exists $\delta_{1} >0 $  such that for all $\by_0 \in B_Y(\delta_{1})$, problem \eqref{eq:NLQprob} possesses a solution $(\bar{\by},\bar{u})$.
Moreover, there exists a constant $M> 0$ independent of $\by_{0}$ such that
\begin{align}\label{eq:sol_est_NLQ}
\max( \| \bar{u} \| _{L^2(0,\infty;U)},\| \bar{\by}\|_{W_\infty}) \le M \|\by_{0}\|_{Y}.
\end{align}
\end{lemma}

\begin{proof}
Let us set, for the moment, $\delta_{1} = \frac{1}{4CM_K^2}$, where $C$ is as in Lemma \ref{lemma:Lipschitz_F} and $M_K$ denotes the constant from Corollary \ref{cor:feas_control}. Applying this corollary (with $\bff=0$), we obtain that for $\by_0 \in B_Y(\delta_1)$, there exists a control $u \in L^2(0,\infty;U)$ with associated state $\by$ satisfying
\begin{align*}
\max(  \| u \|_{L^2(0,\infty;U)},\|\by \|_{W_\infty} ) \leq  M\| \by_{0} \|_{Y},
\end{align*}
where $M=2 M_K \max(1, \|K\|_{\mathcal{L}(Y)})$.
We can thus consider a minimizing sequence $(\by_n,u_n)_{n \in \mathbb N}$ with $J(\by_n,u_n) \leq M^2 \| \by_0\|_Y^2 (1+\alpha)$. We therefore have for all $n \in \mathbb{N}$ that
\begin{equation*}
\| \by_n \|_{L^2(0,\infty;Y)} \leq \sqrt{2}M \|\by_0\|_Y \sqrt{1+ \alpha} \quad \text{and} \quad
\| u_n \|_{L^2(0,\infty;U)} \leq \sqrt{2} M \|\by_0\|_Y \frac{\sqrt{1+ \alpha}}{\sqrt{\alpha}}.
\end{equation*}
Possibly after further reduction of $\delta_{1}$, we eventually obtain that
\begin{align*}
\| \by_0\|_Y + \| \lambda \by_n +Bu_n\|_{L^2(0,\infty;Y)}\leq
\left[ 1 + M \sqrt{2(1 + \alpha)} \left( \lambda + \frac{\| B \|_{\mathcal{L}(U,Y)}}{\sqrt{\alpha}} \right) \right] \delta_1
\leq \frac{1}{4CM_\lambda^2},
\end{align*}
where $M_\lambda$ is as in Corollary \ref{cor:more_regularity}. It then follows that the sequence $(\by_n)_{n \in \mathbb N}$ is bounded in $W_\infty$ with $\sup_{n\in \mathbb N} \| \by_n \|_{W_{\infty}} \leq 2M_\lambda \|\by_{0}\|_Y$. Extracting if necessary a subsequence, there exists $(\bar{\by},\bar{u}) \in W_\infty \times L^2(0,\infty;U)$ such that $(\by_n,u_n) \rightharpoonup (\bar{\by},\bar{u}) \in W_\infty \times L^2(0,\infty;U)$, and $(\bar{\by},\bar{u})$ satisfies \eqref{eq:sol_est_NLQ}.

Let us prove that $(\bar{\by},\bar{u})$ is feasible and optimal.
For any $T>0$ let us consider an arbitrary $\bz \in H^{1}(0,T;V)$. For all $n \in \mathbb N$, we have
\begin{align}\label{eq:limit}
 \int_0^T \left\langle \dot{\by}_n(t),\bz(t) \right\rangle _{V',V} \mathrm{d}t =
 \int_0^T \langle A\by _n (t) - F(\by_n(t))+B u_n(t),\bz(t) \rangle _{V',V}\, \mathrm{d}t.
\end{align}
Since $\dot{\by}_n \rightharpoonup \dot{\bar{\by}}$ in $L^2(0,T;V')$, we can pass to the limit in the l.h.s.\@ of the above equality. Moreover, since $A\by_n \rightharpoonup A\bar{\by} \in L^2(0,T;V')$,
\begin{align*}
 \int_0^T \langle A\by _n (t), \bz(t) \rangle _{V',V} \mathrm{d}t \underset{n \to
\infty}{\longrightarrow}\int_0^T \langle A \bar{\by}(t),\bz(t) \rangle_{V',V}\, \mathrm{d}t.
\end{align*}
Analogously, we obtain that
\begin{align*}
 \int_0^T \langle Bu _n (t), \bz(t) \rangle _{V',V} \mathrm{d}t \underset{n \to
\infty}{\longrightarrow}\int_0^T \langle B\bar{u}(t),\bz(t) \rangle_{V',V}\, \mathrm{d}t.
\end{align*}
We also have
\begin{align*}
&   \left| \int _0^T \langle F(\by_n(t))-F(\bar{\by}(t)),\bz(t) \rangle _{V',V} \, \mathrm{d}t \right| \\
&\qquad  =\left| \int_0^T \langle   N(\by_n(t),\by_n(t))-N(\bar{\by}(t),\bar{\by}(t)), \bz(t) \rangle_{V',V}  \, \mathrm{d}t \right| .
\end{align*}
By Lemma \ref{lem:estimates_oseen_time_var}, it then follows that
\begin{align*}
	&\left| \int _0^T \langle F(\by_n(t))-F(\bar{\by}(t)),\bz(t) \rangle _{V',V} \, \mathrm{d}t \right| \\
	&\quad \le M  \| \bz \|_{L^{\infty}(0,T;V)}   \| \by_n\|_{L^{2}(0,T;Y)}^{\frac{1}{2}} \, \| \by_n - \bar{\by} \|_{L^{2}(0,T;Y)}^{\frac{1}{2}}\,   \| \by_n\|_{L^{2}(0,T;V)}^{\frac{1}{2}} \, \| \by_n - \bar{\by} \|_{L^{2}(0,T;V)}^{\frac{1}{2}}  \\
	 &\qquad +M  \| \bz \|_{L^{\infty}(0,T;V)}   \| \bar{\by}\|_{L^{2}(0,T;Y)}^{\frac{1}{2}} \, \| \by_n - \bar{\by} \|_{L^{2}(0,T;Y)}^{\frac{1}{2}}\,   \| \bar{\by}\|_{L^{2}(0,T;V)}^{\frac{1}{2}} \, \| \by_n - \bar{\by} \|_{L^{2}(0,T;V)}^{\frac{1}{2}}.
\end{align*}
Since $V$ is compactly embedded in $Y$, we obtain that $\| \by_n -\bar{\by}\|_{L^2(0,T;Y)} \underset{n \to \infty}{\longrightarrow} 0$ with the Aubin-Lions lemma. We can pass to the limit in \eqref{eq:limit} and obtain
\begin{align*}
\int_0^T \left\langle \dot{\bar{\by}}(t),\bz(t) \right\rangle _{V',V}\mathrm{d}t= \int_0^T \langle A\bar{\by}  (t) - F(\bar{\by}(t))+B \bar{u}(t),\bz(t) \rangle _{V',V}\, \mathrm{d}t.
\end{align*}
Density of $H^{1}(0,T;V)$ in $L^{2}(0,T;V)$ implies that $e(\bar{\by},\bar{u})=(0,\by_0)$. Finally, by weak lower semi-continuity of norms it follows that $J(\bar{\by},\bar{u}) \leq \liminf_{n \to \infty} J(\by_n,u_n)$, which proves the optimality of $(\bar{\by},\bar{u})$.

Consider now an arbitrary solution $(\tilde{\by},\tilde{u})$ to \eqref{eq:NLQprob}. It then holds that $J(\tilde{\by},\tilde{u})\le M^2\|\by_0\|_Y^2 (1+\alpha)$ from which we obtain that
\begin{align*}
\| \tilde{\by} \|_{L^2(0,\infty;Y)} \le \sqrt{2}M \|\by_0\|_Y \sqrt{1+\alpha} \quad \text{and} \quad \| \tilde{u} \|_{L^2(0,\infty;U)} \le \sqrt{2}M\|\by_0\|_Y \frac{\sqrt{1+\alpha}}{\sqrt{\alpha}}.
\end{align*}
The estimate \eqref{eq:sol_est_NLQ} for $\| \tilde{\by}\|_{W_\infty}$ can now be shown by applying the same arguments as above.
\end{proof}

For the derivation of the optimality system for \eqref{eq:NLQprob} we need the following technical lemma.

\begin{lemma}(\cite[Lemma 2.5]{BreKP18})\label{lem:contractive_pert}
   Let $G \in \mathcal{L}(W_\infty ,L^2(0,\infty;V'))$ be such that $\| G\| < \frac{1}{M_K},$ where $\|G\|$ denotes the operator norm of $G$. Then, for all  $\bff \in L^2(0,\infty;V')$ and $\by_0\in Y$, there exists a unique solution to the following system:
  \begin{align*}
    \dot{\by} = (A-BK)\by(t) + (G\by)(t) + \bff(t), \ \ \by(0)=\by_0.
  \end{align*}
Moreover,
\begin{align*}
  \| \by \| _{W_\infty} \le \frac{M_K}{1-M_K \| G\|} ( \| \bff \| _{L^2(0,\infty;V')} + \| \by_0 \| _Y).
\end{align*}
\end{lemma}

First-order optimality conditions for finite-horizon  optimal control problems have been addressed several times in the literature, we mention e.g. \cite{AT90,MH00,HinK01,IR98}. The  finite-horizon case, and in particular the decay properties
of the state, the costate, and the optimal control,
 require independent treatment, which we provide next.  For an analysis of the linear infinite-horizon problem, we additionally refer to \cite{Ray06}.

\begin{proposition} \label{proposition:optiCondWeak}
There exists $\delta_{2} \in (0,\delta_1]$ such that for all $\by_0 \in B_Y(\delta_2)$, for all solutions $(\bar{\by},\bar{u})$ of \eqref{eq:NLQprob}, there exists a unique costate $\bp \in L^2(0,\infty;V)$ satisfying
\begin{align} \label{eq:costate_NLQ}
-\dot{\bp} - A^{*}\bp - (\bar{\by}\cdot \nabla)\bp+(\nabla \bar{\by})^T \bp  = \ & \bar{\by} \quad \text{(in $(W^0_\infty)'$)}, \\
\label{eq:control_NLQ}
\alpha \bar{u} + B^{*} \bp = \ & 0.
\end{align}
Moreover, there exists a constant $M>0$, independent of $(\bar{\by},\bar{u})$, such that
\begin{equation} \label{eq:estimCostate2}
\| \bp \|_{L^2(0,\infty;V)} \le M \|\by_0\|_Y.
\end{equation}
\end{proposition}

\begin{remark}\label{rem:p_weak_form}
Note that \eqref{eq:costate_NLQ} is a formal expression for
\begin{equation}\label{eq:p_weak_form}
\begin{aligned}
 &\langle -\dot{\bp} - A^{*}\bp - (\bar{\by}\cdot \nabla)\bp+(\nabla \bar{\by})^T \bp  - \bar{\by} , \mathbf{z} \rangle _{(W^0_\infty)',W^0_\infty} \\
 &\quad  = \langle \bp,\dot{\bz}-A\bz+
 (\bz \cdot \nabla)\bar{\by} + (\bar{\by}\cdot \nabla)\bz \rangle_{L^2(0,\infty;V),L^2(0,\infty;V')} - \langle \bar{\by},\bz \rangle_{L^2(0,\infty;Y)} , \ \ \forall \bz \in W_\infty^0,
\end{aligned}
\end{equation}
where $W_\infty^0:= \{ \bz \in W_\infty \,|\, \bz(0)= 0 \}$.
\end{remark}

\begin{proof}[Proof of Proposition \ref{proposition:optiCondWeak}]
Let us set $\delta_2= \delta_1$. By Lemma \ref{lem:ex_opt_sol}, problem \eqref{eq:NLQprob} has a solution $(\bar{\by},\bar{u})$. In the first part of the proof, we derive abstract optimality conditions, by proving that the mapping $e$ (used for formulating the constraints) has a surjective derivative.
For proving the differentiability of $e$, we only need to consider the nonlinear term. We have $F(\by)= N(\by,\by)$ and we know that $N$ is a bounded bilinear mapping from $W_\infty \times W_\infty$ to $L^2(0,\infty;V')$, by Lemma \ref{lem:estimates_oseen_time_var}. Thus $N$ and $F$ are Fr\'echet differentiable, and so is $e$, with
\begin{align*}
 &De(\by,u)\colon W_\infty \times L^2(0,\infty;U) \to L^2(0,\infty;V')\times Y\\
 &De(\by,u)(\bz,v)=(\dot{\bz}-(A\bz-N(\by,\bz)-N(\bz,\by)+Bv),\bz(0)).\end{align*}
Let us show that $De(\bar{\by},\bar{u})$ is surjective if $\delta_{2}$ is sufficiently small.
Let $(\mathbf{r},\mathbf{s})\in L^2(0,\infty;V') \times Y$ and consider the system
\begin{align*}
\dot{\bz} - (A\bz-N(\bar{\by},\bz)-N(\bz,\bar{\by})+Bv) &= \mathbf{r}, \ \
\bz(0)  = \mathbf{s}.
\end{align*}
Observe that by Corollary \ref{cor:estimates_oseen_time_var}
\begin{align*}
 \| N(\bar{\by},\bz)+N(\bz,\bar{\by})\|_{L^2(0,\infty;V')} &\le   M \| \bar{\by}\|_{W_{\infty}}\, \| \bz\|_{W_{\infty}} .
\end{align*}
By Lemma \ref{lem:ex_opt_sol}, it further holds that
\begin{align} \label{eq:estimate_g_aux}
  \| N(\bar{\by},\bz)+N(\bz,\bar{\by})\|_{L^2(0,\infty;V')} \le M \delta_{2} \| \bz\|_{W_\infty}.
\end{align}
For sufficiently small $\delta_{2}$, the operator $G\in \mathcal{L}(W_\infty,L^2(0,\infty;V'))$ defined by
\begin{align}\label{eq:G_aux}
 (G\bz)(t) := D F(\bar \by(t))(\bz(t))= N(\bar{\by}(t),\bz(t))+N(\bz(t),\bar{\by}(t))
\end{align}
satisfies $\|G\| \leq \frac{1}{2M_K} < \frac{1}{M_K}$. By Lemma \ref{lem:contractive_pert} there exists a unique solution $\bz \in W_\infty$  to  the system
\begin{align*}
 \dot{\bz}- (A-BK)\bz+N(\bar{\by},\bz)+N(\bz,\bar{\by}) &= \mathbf{r}, \ \
\bz(0)  = \mathbf{s}.
\end{align*}
Setting $v=-K\bz\in L^2(0,\infty;U)$ proves the surjectivity of $De(\bar{\by},\bar{u})$. Note that
\begin{equation} \label{eq:bound_surjectivity}
\| \bz \|_{W_\infty} \leq M ( \| \mathbf{r} \|_{L^2(0,\infty;V')} + \| \mathbf{s} \|_{L^2(0,\infty;V')} ),
\end{equation}
for some constant $M$ independent of $(\mathbf{r},\mathbf{s})$ and $\by_0$.

From the surjectivity of $De(\bar{\by},\bar{u})$ and Lagrange multiplier theory it follows  that there exists a unique pair $(\bp,\mu)\in L^2(0,\infty;V)\times Y$ such that for all $(\bz,v)\in W_\infty\times L^2(0,\infty;U)$,
\begin{equation}\label{eq:nonlin_aux1}
DJ(\bar{\by},\bar{u})(\bz,v) - \langle (\bp,\mu), De(\bar{\by},\bar{u})(\bz,v) \rangle_{L^2(0,\infty;V)\times Y,L^2(0,\infty;V')\times Y}=0.
\end{equation}
Using \eqref{eq:nonlin_aux1} we derive in the second part of the proof the costate equation \eqref{eq:costate_NLQ} and relation \eqref{eq:control_NLQ}.
As can be easily verified, $J$ is differentiable with
\begin{equation}\label{eq:nonlin_aux2}
 DJ(\bar{\by},\bar{u})(\bz,v) = \langle \bar{\by},\bz \rangle_{L^2(0,\infty;Y)} + \alpha \langle \bar{u},v \rangle_{L^2(0,\infty;U)}
\end{equation}
Moreover, for all $(\bz,v)\in W_\infty  \times L^2(0,\infty;U)$
\begin{equation}\label{eq:nonlin_aux3}
\begin{aligned}
  &\langle (\bp,\mu),De(\bar{\by},\bar{u})(\bz,v) \rangle_{L^2(0,\infty;V)\times Y,L^2(0,\infty;V')\times Y} \\&\qquad = \langle \bp,\dot{\bz}\rangle_{L^2(0,\infty;V),L^{2}(0,\infty;V')} - \langle \bp,A \bz - G \bz \rangle _{L^2(0,\infty;V),L^{2}(0,\infty;V')} \\
  &\qquad \qquad- \langle \bp,Bv \rangle_{L^2(0,\infty;Y)}  + \langle \mu, \bz(0) \rangle_{Y}.
\end{aligned}
\end{equation}
Taking $\bz=0$ and letting $v$ vary in $L^2(0,\infty;U)$, we deduce from \eqref{eq:nonlin_aux1}, \eqref{eq:nonlin_aux2} and \eqref{eq:nonlin_aux3} that
\begin{align*}
 \alpha \bar{u} + B^{*} \bp = 0 \text{ in } L^2(0,\infty;U),
\end{align*}
which proves \eqref{eq:control_NLQ}.
Taking now $v=0$, we obtain that
\begin{equation}\label{eq:kk2}
 \langle \bp, \dot{\bz} \rangle_{L^2(0,\infty;V),L^2(0,\infty;V')} = \langle \bp,A \bz -G \bz\rangle_{L^2(0,\infty;V),L^{2}(0,\infty;V')} + \langle \bar \by, \bz\rangle_{L^2(0,\infty;Y)}, \ \ \forall \bz \in W_\infty^0.
\end{equation}

It remains to bound $\bp$ in $L^2(0,\infty;V)$. Let $\mathbf{r} \in L^2(0,\infty;V')$ and let $(\mathbf{z},v)$ satisfy $De(\bar{\by},\bar{u})(\mathbf{z},v)= (\mathbf{r},0)$ and the bound \eqref{eq:bound_surjectivity} (with $\mathbf{s}= 0$). Using the optimality condition \eqref{eq:nonlin_aux1}, the expression \eqref{eq:nonlin_aux2} of $DJ(\bar{\by},\bar{u})$, estimate \eqref{eq:bound_surjectivity}, and estimate \eqref{eq:sol_est_NLQ} on $(\bar{\by},\bar{u})$, we obtain the following inequalities:
\begin{align*}
&\langle \bp, \mathbf{r} \rangle_{L^2(0,\infty;V),L^2(0,\infty;V')}
=  \ \langle (\bp,\mu), (\mathbf{r},0) \rangle_{L^2(0,\infty;V) \times Y,L^2(0,\infty;V') \times Y} \\
&\quad =  \ \langle De(\bar{\by},\bar{u})'(\bp,\mu), (\bz,v) \rangle_{W_\infty' \times L^2(0,\infty;U),W_\infty \times L^2(0,\infty;U)} \\
  & \quad = \ DJ(\bar{\by},\bar{u})(\bz,v) \\
 &  \quad \leq\ M( \| \bar{\by} \|_{L^2(0,\infty;Y)} + \| \bar{u} \|_{L^2(0,\infty;U)} ) (\| \bz \|_{L^2(0,\infty;Y)} + \| v \|_{L^2(0,\infty;U)} ) \\
 & \quad \leq \ M \| \by_0 \|_{Y} \| \mathbf{r} \|_{L^2(0,\infty;V')}.
\end{align*}
Since $\mathbf{r}$ was arbitrary and since $M$ does not depend on $\mathbf{r}$, we obtain that $\| \bp \|_{L^2(0,\infty;V)} \leq M \| \by_0 \|_{Y}$.
\end{proof}

\subsection{Sensitivity analysis}

We define a mapping $\Phi$ via
\begin{equation}\label{eq:defPhi}
\begin{aligned}
\Phi\colon W_\infty \times  L^2(0,\infty;U) \times &L^2(0,\infty;V)\to  Y \times L^2(0,\infty;V') \times (W^0_\infty)' \times L^2(0,\infty;U)=:X, \\[1ex]
\Phi(\by,u,\bp)& =
\begin{pmatrix}
\by(0) \\ \dot{\by}- A\by + F(\by)-Bu \\ - \dot{\bp} - A^{*}\bp -(\by\cdot \nabla)\bp+(\nabla \by)^T \bp - \by \\
\alpha u + B^{*}\bp
\end{pmatrix},
\end{aligned}
\end{equation}
where the third line again has to be understood formally, see Remark \ref{rem:p_weak_form}.
We endow the space $X$ with the $l_\infty$-product norm.
The well-posedness of $\Phi$ follows from the considerations on $e(\by,u)$ and the costate equation \eqref{eq:costate_NLQ} that have been given in the proof of Proposition \ref{proposition:optiCondWeak}.

%

\begin{lemma} \label{lemma:inverseMappingWeak}
There exist $\delta_3 > 0$, $\delta_3'>0$, and three $C^\infty$-mappings
\begin{equation*}
\by_0 \in B_Y(\delta_3) \mapsto \big( \mathcal{Y}(\by_0),\mathcal{U}(\by_0),\mathcal{P}(\by_0) \big)
\in W_\infty \times L^2(0,\infty;U) \times L^2(0,\infty;V)
\end{equation*}
such that for all $\by_0 \in B_Y(\delta_3)$, the triplet $\big( \mathcal{Y}(\by_0),\mathcal{U}(\by_0),\mathcal{P}(\by_0) \big)$ is the unique solution to
\begin{equation}\label{eq:inverse_aux}
\Phi(\by,u,\bp)= (\by_0,0,0,0), \quad
\max \big( \| \by \|_{W_\infty}, \| u \|_{L^2(0,\infty;U)}, \| \bp \|_{L^2(0,\infty;V)} \big) \leq \delta_3'
\end{equation}
in $W_\infty \times L^2(0,\infty;U) \times L^2(0,\infty;V)$.
Moreover, there exists a constant $M>0$ such that for all $\by_0 \in B_Y(\delta_3)$,
\begin{equation} \label{eq:lipschitzReg}
\max \big( \| \mathcal{Y}(\by_0) \|_{W_\infty}, \| \mathcal{U}(\by_0) \|_{L^2(0,\infty;U)}, \| \mathcal{P}(\by_0) \|_{L^2(0,\infty;V)} \big) \leq M \| \by_0 \|_Y.
\end{equation}
\end{lemma}

\begin{proof}
The result is a consequence of the inverse function theorem. Since $\Phi$ contains only linear terms and three bilinear terms, it is infinitely differentiable. We also have $\Phi(0,0,0)= (0,0,0,0)$. It remains to prove that $D\Phi(0,0,0)$ is an isomorphism. Let $(\bw_1,\bw_2,\bw_3,w_4) \in X$ and let $(\by,u,\bp) \in W_\infty \times L^2(0,\infty;U) \times L^2(0,\infty;V)$.
We have the following equivalence
\begin{equation} \label{eq:equivalenceWeak}
D\Phi(0,0,0) (\by,u,\bp)= (\bw_1,\bw_{2},\bw_{3},w_4) \Longleftrightarrow
\begin{cases}
\begin{array}{rcl}
\by(0) & = & \bw_1 \\
\dot{\by} - A \by - Bu & = & \bw_2 \\
-\dot{\bp} - A^{*} \bp - \by & = & \bw_3 \\
\alpha u + B^{*} \bp & = & w_4.
\end{array}
\end{cases}
\end{equation}
It can be proved with the same techniques as for  \cite[Proposition 3.1, Lemma 4.4]{BreKP18} that the linear system on the left-hand side has a unique solution $(\by,u,\bp)$, moreover,
\begin{equation*}
\| (\by,u,\bp) \|_{W_\infty \times L^2(0,\infty;U) \times L^2(0,\infty;V)}
\leq  M \| (\bw_1,\bw_2,\bw_3,w_4) \|_X .
\end{equation*}
This proves that $D\Phi(0,0,0)$ is an isomorphism. The inverse function theorem ensures the existence of $\delta_3>0$, $\delta_3'>0$, and $C^\infty$-mappings $\mathcal{Y}$, $\mathcal{U}$, and $\mathcal{P}$ with the properties announced in  \eqref{eq:inverse_aux}.

It remains to prove \eqref{eq:lipschitzReg}. Reducing if necessary $\delta_3$, we can assume that the norms of the derivatives of the three mappings are bounded on $B_Y(\delta_3)$ by some constant $M>0$. The three mappings are therefore Lipschitz continuous with modulus $M$. Estimate \eqref{eq:lipschitzReg} follows, since $\big( \mathcal{Y}(0),\mathcal{U}(0),(\mathcal{P}(0) \big)= (0,0,0)$.
\end{proof}


\begin{proposition} \label{proposition:UisOptimal}
There exists $\delta_4 \in (0,\min(\delta_2,\delta_3)]$ such that for all $\by_0 \in B_Y(\delta_4)$, the pair $(\mathcal{Y}(y_0),\mathcal{U}(y_0))$ is the unique solution to \eqref{eq:NLQprob} with initial condition $\by_0$. Moreover, $\mathcal{P}(\by_0)$ is the unique associated costate.
\end{proposition}

\begin{proof}
Let us set $\delta_4= \min(\delta_2,\delta_3)$ for the moment. Let $\by_0 \in B_Y(\delta_4)$. By Lemma \ref{lem:ex_opt_sol} and Proposition \ref{proposition:optiCondWeak}, there exist a solution $(\bar{\by},\bar{u})$ to \eqref{eq:NLQprob} with associated costate $\bar{\bp}$ which necessarily satisfies
\begin{align*}
\max(\| \bar{\by} \|_{W_{\infty}}, \| \bar{u} \|_{L^{2}(0,\infty;U)}, \|\bar{\bp}\|_{L^2(0,\infty;V)} ) \leq M \| \by_{0} \|_{Y}.
\end{align*}
By further reduction of $\delta_{4}$, we obtain that
\begin{equation*}
\max(\| \bar{\by} \|_{W_{\infty}}, \| \bar{u} \|_{L^{2}(0,\infty;U)}, \|\bar{\bp}\|_{L^2(0,\infty;V)} ) \leq \delta_3'.
\end{equation*}
Since $\Phi(\bar{\by},\bar{u},\bar{\bp})= (\by_0,0,0,0)$, Lemma \ref{lemma:inverseMappingWeak} implies that $(\bar{\by},\bar{u},\bar{\bp})= (\mathcal{Y}(\by_0),\mathcal{U}(\by_0),\mathcal{P}(\by_0))$. The proposition is proved.
\end{proof}

\begin{corollary} \label{coro:diff}
The value function $\mathcal{V}$ is infinitely differentiable on $B_Y(\delta_4)$.
\end{corollary}

 \begin{proof}
 The cost function $J$ is clearly infinitely differentiable. Since $\mathcal{V}(\by_0)= J(\mathcal{Y}(\by_0),\mathcal{U}(\by_0))$, $\mathcal{V}$ is then the composition of infinitely differentiable mappings, which shows the assertion.
 \end{proof}

 \subsection{Additional regularity for $\bp$}

We next assert that for small initial data $\by_0$ the adjoint state is more regular than $\bp\in L^2(0,\infty;V)$. For this, we need more smoothness of the boundary $\Gamma$.

\begin{customthmm}{A3}\label{ass:A3}
 Let $\Omega \subset \mathbb R^2$ denote a bounded domain with smooth boundary $\Gamma$.
 \end{customthmm}

 \begin{proposition} \label{proposition:more_reg}
There exists $\tilde{\delta}_{4}\in (0,\delta_4] $ such that for all $\by_0 \in B_Y(\tilde{\delta}_4)$, for all solutions $(\bar{\by},\bar{u})$ of
(P), there exists a unique costate $\bp \in W_\infty$ satisfying
\begin{align} \label{eq:costate_NLQ_improved}
-\dot{\bp} - A^{*}\bp - (\bar{\by}\cdot \nabla)\bp+(\nabla \bar{\by})^T \bp  = \ & \bar{\by} \quad \text{(in $L^2(0,\infty;V')$)}.
\end{align}
Moreover, there exists a constant $M>0$, independent of $(\bar{\by},\bar{u})$, such that
\begin{equation} \label{eq:estimCostate2_bis}
\| \bp \|_{W_\infty} \le M \|\by_0\|_Y.
\end{equation}
\end{proposition}
The proof is given in the appendix.

\section{Derivatives of the value function}

By standard arguments, we can derive a Hamilton-Jacobi-Bellman equation which provides an optimal feedback control based on the derivative of the value function.

All along the section, the first-order derivative $D\mV(\by_0)$ is either seen as a linear form on $Y$ or is identified with its Riesz representative in $Y$. The identification is done for example in the term $\| B^{*} D\mV(\by_0) \|_U^2$ appearing in the HJB equation below.

\begin{proposition} \label{prop:hjb_form}
There exists $\delta_5 \in (0,\tilde{\delta_4}]$ such that for all $\by_0 \in B_{Y}(\delta_5) \cap \mD(A)$, the following Hamilton-Jacobi-Bellman equation holds:
\begin{equation} \label{eq:HJB}
  D\mV(\by_0) (A \by_0 - F(\by_0) ) + \frac{1}{2} \| \by_0 \|_Y^2 - \frac{1}{2\alpha} \| B^{*} D\mV(\by_0) \|_U^2=0.
\end{equation}
Moreover,
\begin{equation} \label{eq:optimFeedback}
\bar{u}(t)= -\frac{1}{\alpha} B^{*}D\mathcal{V}(\bar{\by}(t)), \quad \text{for all $t \geq 0$},
\end{equation}
where $(\bar{\by},\bar{u})= (\mathcal{Y}(\by_0),\mathcal{U}(\by_0)).$
\end{proposition}

\begin{remark}
  Note that by, e.g., \cite[Proposition 1.7]{Bar11}, we have that $F\colon \mD(A)\times \mD(A) \to Y$ and, as a consequence, the term $D\mV(\by_{0}) F(\by_{0})$ is well-defined.
\end{remark}

\begin{proof}
Let us set $\delta_5= \delta_4$. Let $\by_0 \in B_Y(\delta_5) \cap \mD(A)$.
Let us consider the Hamiltonian of the system, defined by
\begin{equation*}
H(\by_0,u,\bp)= \frac{1}{2} \| \by \|_Y^2 + \frac{\alpha}{2} \| u \|_U^2 + \langle \bp, A \by - F(\by) + Bu \rangle_Y, \ \ \forall (\by,u,\bp) \in \mD(A) \times U \times Y.
\end{equation*}
Using the arguments provided in the proof of \cite[Proposition 9]{BreKP19}, one can prove that
\begin{equation*}
\min_{u \in U} H(\by_0,u,D \mathcal{V}(\by_0))= 0,
\end{equation*}
from which \eqref{eq:HJB} derives. One can also prove that
$$\bar{u}(0)= \text{arg min}_{u \in U} H(\by_0,u,D\mV(\by_0)),$$
 which proves \eqref{eq:optimFeedback} for $t=0$.
Let us emphasize that the assumptions which are required in \cite[Proposition 9]{BreKP19} are satisfied. In particular,
the optimality condition $\bar{u}(t)= -\frac{1}{\alpha} B^* \bar{\bp}(t)$ which holds in $L^2(0,\infty;U)$ implies that $\bar{u}$ is almost everywhere equal to a continuous function. We can thus assume that $\bar{u}$ is continuous.
 For proving \eqref{eq:optimFeedback} for all $t \geq 0$, one has first to reduce $\delta_5$ so that $\| \bar{\by}(t) \|_Y \leq \delta_4$, for all $t \geq 0$. For a given $t \geq 0$, we have by dynamic programming that $(\bar{\by}(t+ \cdot),\bar{u}(t+\cdot))$ is the solution to \eqref{eq:NLQprob} with initial condition $\bar{\by}(t)$ and thus \eqref{eq:optimFeedback} holds true at $t$.
\end{proof}


For deriving a Taylor series expansion of $\mV$, let us follow the approach from \cite{Alb61} and differentiate \eqref{eq:HJB} in some direction $\bz_{1} \in \mD(A)$. To alleviate the calculations, we denote the variable $\by_0$ in \eqref{eq:HJB} by $\by$. We then obtain
\begin{align*}
 &D^{2}\mV(\by)\left(A\by-F(\by),\bz_{1}\right)+D\mV(\by)\left(A\bz_{1}-A_{0}(\by,\bz_{1})\right)	+ \langle \by,\bz_{1} \rangle_{Y}\\
 &\qquad - \frac{1}{\alpha} \langle B^{*}D^{2}\mV(\by)(\cdot,\bz_{1}),B^{*}D\mV(\by) \rangle_{U}=0.
\end{align*}
 A second differentiation in the directions $(\bz_{1},\bz_{2}) \in \mD(A)^2$ yields the equation
 \begin{align*}
&	 D^{3}\mV(\by)\left(A\by-F(\by),\bz_{1},\bz_{2}\right)+D^{2}\mV(\by)\left(A\bz_{2}-A_{0}(\by,\bz_{2}),\bz_{1}\right)+D^{2}\mV(\by)\left(A\bz_{1}-A_{0}(\by,\bz_{1}),\bz_{2}\right) \\
&\qquad -D\mV(\by)\left(A_{0}(\bz_{2},\bz_{1})\right) + \langle \bz_{2},\bz_{1} \rangle_{Y} -\frac{1}{\alpha}\langle B^{*} D^{3}\mV(\by)\left(\cdot,\bz_{1},\bz_{2}\right),B^{*} D \mV(\by) \rangle_{U}\\
&\qquad  - \frac{1}{\alpha} \langle B^{*} D^{2}\mV(\by)\left(\cdot,\bz_{1}\right), B^{*} D^{2}\mV(\by)\left(\cdot,\bz_{2}\right) \rangle _{U} = 0.
 \end{align*}
Since $\mV(0)=0$ and $\mV(\by)\ge 0$ for all $\by \in Y$, it follows that $D\mV(0)=0$. We can thus evaluate the last equation for $\by=0$ to obtain
\begin{equation}\label{eq:are_mlf}
\begin{aligned}
  & D^{2}\mV(0)\left(A\bz_{2},\bz_{1}\right)+D^{2}\mV(0)\left(A\bz_{1},\bz_{2}\right)   + \langle \bz_{2},\bz_{1} \rangle_{Y}  \\
  &\qquad\qquad   - \frac{1}{\alpha} \langle B^{*} D^{2}\mV(0)\left(\cdot,\bz_{1}\right), B^{*} D^{2}\mV(0)\left(\cdot,\bz_{2}\right) \rangle _{U} = 0.
\end{aligned}
\end{equation}
We recall that $D^{2}\mV(0)\in \mathcal{M}(Y\times Y,\mathbb R)$ is a bounded and symmetric bilinear form on $Y$ and thus can be represented (see, e.g., \cite[Chapter 5, Section 2]{Kat80}) by an operator $\Pi \in \mL(Y)$ such that
\begin{align*}
	D^{2}\mV(0)(\by,\bz) = \langle \Pi \by ,\bz \rangle _{Y}, \ \ \text{for all } \by,\bz \in Y.
\end{align*}
As a consequence, we can formulate \eqref{eq:are_mlf} as
\begin{align}\label{eq:are}
 \langle \bz_{2},A^{*} \Pi \bz_{1 } \rangle_{Y} + \langle \Pi A\bz_{1},\bz_{2}\rangle _{Y}+\langle \bz_{2},\bz_{1} \rangle_{Y} - \frac{1}{\alpha} \langle B^{*} \Pi \bz_{1},B^{*}\Pi \bz_{2} \rangle_{U} = 0.
\end{align}
Equation \eqref{eq:are} is the well-known \emph{algebraic operator Riccati equation} which has been studied in detail in, e.g., \cite{CurZ95,LasT00}. From the stabilizability assumption \ref{ass:A2}, and the fact that the pair $(A,\mathrm{id})$ is  exponentially detectable as a consequence of \eqref{eq:VY-coercivity}, we conclude that \eqref{eq:are} has a unique stabilizing solution $\Pi \in \mathcal{L}(Y)$.
 In the discussion below, we denote by
$$A_{\pi} := A-\frac{1}{\alpha}BB^{*}\Pi$$
the closed-loop operator associated with the linearized stabilization problem. In particular, let us mention that $A_{\pi}$ generates an analytic exponentially stable semigroup $e^{A_{\pi}t}$ on $Y$. Hence, for trajectories of the form $\tilde{\by}=e^{A\cdot}\by$, $\by\in Y$ it follows that $\tilde{\by}\in W_{\infty}$.

For higher order derivatives of $\mV$, we follow the exposition from \cite{BreKP19}. For this purpose, let us briefly recall the symmetrization technique introduced there. Let $i$ and $j \in \mathbb{N}$, consider
\begin{equation*}
S_{i,j}= \big\{ \sigma \in S_{i+j} \,|\, \sigma(1)< \dots < \sigma(i) \text{ and }
\sigma(i+1) < \dots < \sigma(i+j) \big \},
\end{equation*}
where $S_{i+j}$ is the set of permutations of $\{ 1,\dots,i+j \}$.
A permutation $\sigma \in S_{i,j}$ is uniquely defined by the subset
$\{\sigma(1),\dots,\sigma(i) \}$, therefore, the cardinality of $S_{i,j}$ is equal
to the number of subsets of cardinality $i$ of $\{ 1,\dots,i+j\}$, that is to say
$|S_{i,j}|= \binom{i+j}{i}$.
For a multilinear mapping $\mathcal{T}$ of order $i+j$, we set
\begin{equation} \label{eq:defSym}
\text{Sym}_{i,j}(\mathcal{T})(\bz_1,\dots,\bz_{i+j}) =
{\binom{i+j}{i}}^{-1} \Big[ \sum_{\sigma \in S_{i,j}} \mathcal{T}
(\bz_{\sigma(1)},\dots,\bz_{\sigma(i+j)} ) \Big].
\end{equation}
The following proposition is a generalization of the Leibniz formula for the differentiation of the product of two functions.

\begin{proposition}\label{prop:leibniz}
Let $Z$ be a Hilbert space.
Let $f\colon Y\to Z$ and $g\colon Y\to Z$ be two $k$-times continuously differentiable functions. Then, for all $k\ge 1$, for all $\by\in Y$ and $(\bz_1,\dots,\bz_k)\in Y^k$,
\begin{align*}
 D^k [\langle f(\by) , g(\by) \rangle_Z](\bz_1,\dots,\bz_k) =
  \sum_{i=0}^k \binom{k}{i} \Symm_{i,k-i} (D^i f(\by)\otimes D^{k-i}g(\by) )(\bz_{1},\dots,\bz_{k}).
 \end{align*}
\end{proposition}

\begin{proof}
The proof is analogous to the one given in \cite[Lemma 10]{BreKP19} for $Z= \R$.
\end{proof}

\begin{theorem} \label{thm:DifferentiabilityImpliesLyapunov}
Let $k \geq 3$. For all $\bz_1,\dots,\bz_k \in \mathcal{D}(A)$,
\begin{equation} \label{eq:Lyapunov1}
\sum_{i=1}^k
\mathcal{D}^k\mathcal{V}(0)(\bz_1,\dots,\bz_{i-1},A_{\pi}\bz_i,\bz_{i+1},\dots,\bz_k)
=  \mathcal{R}_k(\bz_1,\dots,\bz_k),
\end{equation}
where the multilinear form $\mathcal{R}_{k}\colon \mD(A)^k \rightarrow \R$ is given by
\begin{align*}
& \mathcal{R}_{k}(\bz_1,\dots,\bz_k)=\frac{1}{2\alpha} \sum_{i=2}^{k-2} \begin{pmatrix} k \\ i \end{pmatrix}
\Symm\limits_{i,k-i} \big( \mathcal{C}_i \otimes
\mathcal{C}_{k-i} \big)(\bz_1,\dots,\bz_k) \\
&\hspace{2.8cm} 
+ \frac{k(k-1)}{2} \Symm\limits_{k-2,2}\big( D^{k-1}\mV(0)\otimes
D^{2}F(0) \big)(\bz_1,\dots,\bz_k)
\end{align*}
with $
 \mathcal{C}_{i}(\bz_1,\dots,\bz_i) =  \displaystyle{B^{*} D^{i+1}
\mathcal{V}(0)(\cdot,\bz_1,\dots,\bz_i)}   $ and $D^{2}F(0)(\bz_{1},\bz_{2})=A_{0}(\bz_{1},\bz_{2})$.
\end{theorem}

\begin{proof}
The proof relies on successive differentiations of \eqref{eq:HJB}. For a bilinear control problem, a similar result has been obtained in \cite[Theorem 12]{BreKP19}.
In particular, it was shown that
\begin{align}\label{eq:der_aterm}
 \left( D^k [\mV(\by)(A\by)]_{\by=0}\right)(\bz_1,\dots,\bz_k) = \sum_{i=1}^k D^k \mV(0)(\bz_1,\dots,\bz_{i-1},A\bz_i,\bz_{i+1},\dots,\bz_k).
\end{align}
Obviously, for $k \geq 3$,  we have $D^k(\frac{1}{2}\|\by\|_Y^2)=0$. Let us discuss the structure of the derivatives of the remaining terms appearing in \eqref{eq:HJB}. Applying Proposition \ref{prop:leibniz} to the term $\| B^{*}  D\mV (\by) \|_U^2 $, we obtain
  \begin{align*}
  D^k \| B^{*} D\mV (\by) \|_U^2 = \sum_{i=0}^k \begin{pmatrix} k \\ i \end{pmatrix} \Symm_{i,k-i} (D^i (B^{*} D\mV(\by)) \otimes D^{k-i}(B^{*} D \mV(\by))) .
  \end{align*}
Since $\mV$ has a minimum at the origin,  we have $D\mV(0)=0$ and the terms for $i=0$ and $i=k$ vanish when evaluated in $\by=0$. By definition of the $\Symm$-operator, for $i=1$ we obtain
\begin{align*}
   & \left .\begin{pmatrix} k \\ 1 \end{pmatrix} \Symm_{1,k-1} (D (B^{*} D\mV(\by)) \otimes D^{k-1}(B^{*}  D \mV(\by))) (\bz_1,\dots,\bz_k)\right|_{\by=0} \\
    &\qquad =\sum_{\sigma\in S_{1,k-1}}\langle B^{*} D^2\mV(0)(\cdot,\bz_{\sigma(1)}),B^{*} D^k \mV(0)(\cdot,\bz_{\sigma(2)},\dots,\bz_{\sigma(k)}) \rangle_U \\
     &\qquad =\sum_{\sigma\in S_{1,k-1}}\langle BB^{*} D^2\mV(0)(\cdot,\bz_{\sigma(1)}),D^k \mV(0)(\cdot,\bz_{\sigma(2)},\dots,\bz_{\sigma(k)}) \rangle_Y \\
     &\qquad =\sum_{\sigma\in S_{1,k-1}} D^k \mV(0)( BB^{*} D^2\mV(0)(\cdot,\bz_{\sigma(1)}),\bz_{\sigma(2)},\dots,\bz_{\sigma(k)})
\end{align*}
As explained previously, we can represent $D^2\mV(0)$ in terms of the solution $\Pi$ of the algebraic operator Riccati equation. This shows
\begin{equation}\label{eq:aux_der_hjb_bterm}
\begin{aligned}
& \left .\begin{pmatrix} k \\ 1 \end{pmatrix} \Symm_{1,k-1} (D (B^{*} D\mV(\by)) \otimes D^{k-1}(B^{*}  D \mV(\by))) (\bz_1,\dots,\bz_k)\right|_{\by=0} \\
&\qquad = \sum_{\sigma\in S_{1,k-1}} D^k \mV(0)( BB^{*}\Pi \bz_{\sigma(1)},\bz_{\sigma(2)},\dots,\bz_{\sigma(k)}) \\
& \qquad = \sum_{i=1}^k D^k \mathcal{V}(0)(\bz_1,\dots,\bz_{i-1},BB^* \bz_{i},\bz_{i+1},\dots,\bz_k).
\end{aligned}
\end{equation}
A similar relation can be derived for $i=k-1$.
Finally we consider the term $D^k(D\mV(\by)F(\by))$. By Proposition \ref{prop:leibniz}, we get
\begin{align*}
& D^{k} (D\mV(\by)F(\by)) (\bz_{1},\dots,\bz_{k}) \\
 &\quad = D^{k} \langle D\mV(\by),F(\by) \rangle_{Y} (\bz_{1},\dots,\bz_{k})\\
& \quad=\sum_{i=0}^{k} \begin{pmatrix} k \\ i \end{pmatrix} \Symm\limits_{i,k-i} \left( D^{i+1}\mV(\by) \otimes D^{k-i}F(\by)\right) (\bz_{1},\dots,\bz_{k}).
\end{align*}
Since $D^{3+\ell}F(\by)=0$ for all $\ell \ge 0$, the previous equation simplifies as follows
\begin{align*}
D^{k} (D\mV(\by)F(\by))(\bz_{1},\dots,\bz_{k}) =\sum_{i=k-2}^{k} \begin{pmatrix} k \\ i \end{pmatrix} \Symm\limits_{i,k-i} \left( D^{i+1}\mV(\by) \otimes D^{k-i}F(\by)\right) (\bz_{1},\dots,\bz_{k}).
\end{align*}
Evaluating the last expression in $\by=0$ yields
\begin{align}\label{eq:aux_der_hjb_nonl}
\left(D^{k} [D\mV(\by)F(\by)]_{\by=0}\right)(\bz_{1},\dots,\bz_{k}) =\frac{k(k-1)}{2}  \Symm\limits_{k-2,2} \left( D^{k-1}\mV(0) \otimes D^{2}F(0)\right) (\bz_{1},\dots,\bz_{k}),
\end{align}
since $F(0)$ and $DF(0)$ are both null.
Combining \eqref{eq:der_aterm}, \eqref{eq:aux_der_hjb_bterm} and \eqref{eq:aux_der_hjb_nonl} proves the assertion.
\end{proof}

\section{Polynomial feedback laws}

\subsection{Estimates for the velocity}

In this section we analyze the polynomial feedback law $u_d$ derived from the Taylor series approximation of the value function
\begin{align*}
\mV_d(\by) := \sum_{k=2}^d \frac{1}{k!}D^{k}\mV(0)(\by,\dots,\by),
\end{align*}
for a given $d \geq 2$. The feedback $u_d \colon Y \rightarrow U$ is obtained by approximating $\mV$ with $\mV_d$ in formula \eqref{eq:optimFeedback}, that is
\begin{align*}
u_{d}(\by) = -\frac{1}{\alpha} B^{*}D\mathcal{V}_d(\by)= - \frac{1}{\alpha} \sum_{k=2}^{d} \frac{1}{(k-1)!} B^{*} D^{k}\mV(0)(\cdot,\by,\dots,\by).
\end{align*}
The associated closed-loop system is given by
\begin{equation} \label{eq:cls}
\dot{\by}_d = A \by_d - F(\by_d) + Bu_{d}(\by_d), \quad \by_d(0) = \by_0.
\end{equation}
Below we will also derive an estimate for the open-loop control, i.e., the function defined by
\begin{align}\label{eq:opvsclosed}
u_d\colon [0,\infty) \to U, \ t\mapsto u_d(t):=u_d(\by_d(t))
\end{align}
 which is obtained via closed-loop dynamcics here. With slight abuse of notation, the open-loop control $u_d(t)$ as well as its closed-loop interpretation $u_d(\by_d(t))$ will both be denoted with $u_d$.

We begin with some local Lipschitz continuity estimates for the nonlinear part of the feedback law. For this purpose, we set
\begin{align}\label{eq:nonl_feed}
 G_k(\by):=-\frac{1}{\alpha (k-1)!}BB^{*}D^{k}\mV(0)(\cdot,\by,\dots,\by),
\end{align}
for all $k \geq 3$. The closed-loop system can be reformulated as follows:
\begin{align}
\dot{\by}_d
& =A_\pi \by_d - F(\by_d) -\frac{1}{\alpha} \sum_{k=3}^d \frac{1}{(k-1)!}BB^{*}D^{k} \mV(0)(\cdot,\by_d,\dots,\by_d) \notag \\
& = A_{\pi} \by_d - F(\by_d) + \sum_{k=3}^{d} G_{k}(\by_d).
\label{eq:loc_stab_cls}
\end{align}

\begin{lemma} \label{lemma:Lipschitz_G}
For all $k \geq 3$, there exists a constant $C(k)>0$ such that for all $\by$ and $\bz \in Y$,
\begin{equation*}
\| G_k(\by)- G_k(\bz) \|_Y \leq C(k) \| \by - \bz \|_Y \max( \| \by \|_Y, \| \bz \|_Y )^{k-2}.
\end{equation*}
Moreover, for all $\delta \in [0,1]$, for all $\by$ and $\bz \in W_\infty$ such that $\| \by \|_{W_\infty} \leq \delta$ and $\| \bz \|_{W_\infty} \leq \delta$,
\begin{equation*}
\| G_k(\by)-G_k(\bz) \|_{L^2(0,\infty;V')}
\leq C(k) \delta \| \by - \bz \|_{W_\infty}.
\end{equation*}
\end{lemma}

\begin{proof}
We have the identity
\begin{align*}
& D^k \mathcal{V}(0)(\cdot, \by,\dots,\by) -
D^k \mathcal{V}(0)(\cdot, \bz,\dots,\bz)
 = D^k \mathcal{V}(0)(\cdot, \by - \bz, \by,\dots,\by) \\
&  \qquad \qquad \qquad + D^k \mathcal{V}(0)(\cdot, \bz, \by - \bz,\dots,\by)
+ \dots
+ D^k \mathcal{V}(0)(\cdot, \bz,\dots,\bz, \by- \bz).
\end{align*}
The first inequality easily follows, with $C(k)= \frac{1}{\alpha (k-2)!} \| B\|_{\mathcal{L}(U,Y)}^2  \|D^k \mathcal{V}(0) \|$ and $\|D^k \mathcal{V}(0) \|$ as defined in \eqref{eq:OperatorNormTensor}. We also obtain that for all $\by$ and $\bz \in W_\infty$,
\begin{equation*}
\| G_k(\by)-G_k(\bz) \|_{L^2(0,\infty;V')}
\leq C(k) \| \by - \bz \|_{W_\infty} \max( \| \by \|_{W_\infty}, \| \bz \|_{W_\infty} )^{k-2}.
\end{equation*}
The second inequality follows, since $k \geq 3$ and $\delta \leq 1$.
\end{proof}

The well-posedness of the closed-loop system can be now established with the same tools as those used in Lemma \ref{lem:eq:non_loc_sol}.

\begin{theorem} \label{thm:cls_well_posed}
Let $d \geq 2$. Let $C$ and $C(k)$ denote the constants from Lemma \ref{lemma:Lipschitz_F} and Lemma \ref{lemma:Lipschitz_G}. There exists a constant $M_{\mathrm{cls}}$ such that for all $\by_0 \in Y$ with
\begin{align*}
\| \by _0\| _Y   \le \frac{1}{4(C+ \sum_{k=3}^d C(k))M_{\mathrm{cls}}^2},
\end{align*}
the closed-loop system \eqref{eq:cls} has a unique solution $\by_d$ in $W_\infty$, which satisfies
\begin{equation} \label{eq:estimate_cls}
\| \by_d\|_{W_\infty}
\leq 2 M_{\mathrm{cls}} \| \by_0 \|_Y.
\end{equation}
\end{theorem}

\begin{proof}
The existence of a solution $\by \in W_\infty$, satisfying \eqref{eq:estimate_cls}, can be obtained exactly as in Lemma \ref{lem:eq:non_loc_sol}. Thus we only discuss uniqueness.
Let $\by$ and $\bz$ denote two solutions to \eqref{eq:cls} in $W_\infty$. Let us set $\mathbf{e}= \by - \bz$. Arguing as in the proof of Lemma \ref{lem:eq:non_loc_sol}, one can prove the existence of $M>0$ such that
\begin{equation*}
\frac{1}{2} \frac{\text{d}}{\text{d} t} \| \mathbf{e} \|_Y^2
\leq M \Big( 1 + \| \by \|_Y^2 \| \by \|_V^2 + \| \bz \|_Y \|^2 \bz \|_V^2 +
\sum_{k=3}^d C(k)^2 \max ( \| \by \|_Y, \| \bz \|_Y )^{2(k-2)} \Big) \| \mathbf{e} \|_Y^2,
\end{equation*}
for all $t \geq 0$.
Since $\by$ and $\bz \in W_\infty$ and $\mathbf{e}(0)=0$, we obtain with Gronwall's inequality that $\mathbf{e}=0$, which proves the uniqueness of the solution to the closed-loop system.
\end{proof}

\begin{theorem}\label{theo19}
Let $d \geq 2$. There exist $\delta_6 > 0$ and $M > 0$ such that for all $\by_{0} \in B_Y(\delta_6)$, it holds that
\begin{align*}
\| \bar{\by}-\by_{d} \|_{W_{\infty}} & \le M \| \by _{0} \|_{Y}^{d}, \\
\max\left( \| \bar{u}-u_{d} \|_{L^{2}(0,\infty;U)},\| \bar{u}-u_{d} \|_{L^{\infty}(0,\infty;U)}\right)& \le  M \| \by _{0} \|_{Y}^{d},
\end{align*}
where $(\bar{\by},\bar{u})= (\mathcal{Y}(\by_0),\mathcal{U}(\by_0))$, $\by_{d}$ is the solution of the closed-loop system \eqref{eq:cls} with initial condition $\by_0$, and $u_d$ is as defined in \eqref{eq:opvsclosed}.
\end{theorem}

\begin{proof}
Let us fix $\delta_6= \min \big( \delta_5,(4(C+ \sum_{k=3}^d C(k)) M_{\text{cls}}^2)^{-1} \big)$, so that Proposition \ref{prop:hjb_form} and Theorem \ref{thm:cls_well_posed} apply for $\by_0 \in B_Y(\delta_6)$.
By Taylor's theorem, see, e.g., \cite[Theorem 4A]{Zei86}, there exists $\delta > 0$ such that for all $\by \in B_Y(\delta)$,
\begin{equation} \label{eq:taylor_expansion}
D\mV({\by}) = \sum_{k=2}^{d}\frac{1}{(k-1)!} D^{k} \mV(0) (\cdot,{\by},\dots,{\by}) + R_{d} ({\by}),
\end{equation}
where the remainder term $R_{d}$ satisfies
\begin{align*}
\| R_{d}({\by}) \|_{Y} \leq M \| {\by} \|^{d}_{Y},
\end{align*}
for some constant $M$ independent of $\by$. Reducing if necessary $\delta_6$, we have that $\| \bar{\by}(t) \|_Y \leq \delta$ for all $t \geq 0$. Combining then \eqref{eq:optimFeedback} and the Taylor expansion \eqref{eq:taylor_expansion}, we obtain that
\begin{equation} \label{eq:abstract_opt_traj}
\dot{\bar{\by}}
= A \bar{\by} - F(\bar{\by}) - \frac{1}{\alpha} B B^* D \mathcal{V}(\bar{\by})
= A_\pi \bar{\by} - F(\bar{\by})
+ \sum_{k=3}^d G_k(\bar{\by}) - \frac{1}{\alpha} B B^* R_d(\bar{\by}).
\end{equation}
Let us now consider the error dynamics $\mathbf{e}:= \bar{\by}-\by_{d}$. We have $\mathbf{e}(0)=0$, moreover by \eqref{eq:loc_stab_cls} and \eqref{eq:abstract_opt_traj},
\begin{align*}
 \dot{ \mathbf{e} } & = A_\pi \mathbf{e} - F(\bar{\by})+F(\by_{d}) + \sum_{k=3}^{d} (G_{k}(\bar{\by})-G_{k}(\by_{d})) - \frac{1}{\alpha} B B^* R_d(\bar{\by}).
\end{align*}
Alternatively, $\mathbf{e}$ can be expressed as the solution of the system
\begin{align}\label{eq:aux_error}
\dot{\mathbf{e}}= A_{\pi} \mathbf{e} + \mathbf{f}, \quad \mathbf{e}(0)=0,
\end{align}
where the source term $\mathbf{f}$ is given by
\begin{align*}
\mathbf{f}=
-F(\bar{\by})+F(\by_{d})
+ \sum_{k=3}^{d} (G_{k}(\bar{\by})-G_{k}(\by_{d}))
- \frac{1}{\alpha} B B^* R_d(\bar{\by}).
\end{align*}
Consider $\tilde{\delta} \in (0,1]$. The precise value of $\tilde{\delta}$ will be fixed later.
By Lemma \ref{lemma:inverseMappingWeak} and Theorem \ref{thm:cls_well_posed}, we can reduce $\delta_6$ so that $\max(\| \bar{\by} \|_{W_{\infty}},\| \by_{d} \|_{W_{\infty}} ) \leq \tilde{\delta}$.
We first observe that
\begin{equation*}
\Big\| \frac{1}{\alpha} B B^* R_d(\bar{\by}) \Big\|_{L^2(0,\infty;V')}
\leq M \| \bar{\by} \|_{L^\infty(0,\infty;Y)}^{d-1} \| \bar{\by} \|_{L^2(0,\infty;Y)}
\leq M \| \by_0 \|_Y^d.
\end{equation*}
Applying further Lemma \ref{lemma:Lipschitz_F} and Lemma \ref{lemma:Lipschitz_G}, we obtain
\begin{align*}
\|\mathbf{f}\|_{L^2(0,\infty;V')}
& \leq M \Big(\| F(\bar{\by})-F(\by_{d})\|_{L^2(0,\infty;V')}
+ \sum_{k=3}^{d} \| G_{k}(\bar{\by})-G_{k}(\by_{d})\|_{L^2(0,\infty;V')}
+ \| \by_0 \|_Y^d \Big) \\
& \leq M (\tilde{\delta} \| \mathbf{e} \|_{W_{\infty}} + \| \by_{0} \|_{Y}^{d} ).
\end{align*}
For the solution of system \eqref{eq:aux_error} we thus obtain the estimate
\begin{align*}
\| \mathbf{e} \|_{W_{\infty}} \le M \|\mathbf{f}\|_{L^2(0,\infty;V')} \le M (\tilde{\delta} \| \mathbf{e} \|_{W_{\infty}} + \| \by_{0} \|_{Y}^{d} ).
\end{align*}
The constant $M> 0$ in the above estimate is independent of $\tilde{\delta}$. We can now define $\tilde{\delta}= \min \big( 1, \frac{1}{2M} \big)$. The first estimate on $\| \bar{\by}-\by_{d} \|_{W_{\infty}}$ follows.

Let us estimate $\bar{u}-u_d$. By \eqref{eq:optimFeedback} and by definition of the generated open-loop control $u_d$, we have that
\begin{align*}
\bar{u}- u_d
= - \frac{1}{\alpha} B^* \big( D \mathcal{V}(\bar{\by})- D \mathcal{V}_d(\by_d) \big)
= - \frac{1}{\alpha} B^* \big ( R_d(\bar{\by}) + D \mathcal{V}_d(\bar{\by})- D \mathcal{V}_d(\by_d) \big).
\end{align*}
Let us estimate the two terms of the right-hand side.
It is easy to check that
\begin{equation*}
\max \big(
\| R_d(\bar{\by}) \|_{L^\infty(0,\infty;Y)},
\| R_d(\bar{\by}) \|_{L^2(0,\infty;Y)} \big) \leq M \| \by_0 \|_Y^d.
\end{equation*}
Using the techniques of Lemma \ref{lemma:Lipschitz_G} and the estimate on $\| \bar{\by} - \by_d \|_{W_\infty}$, we also obtain that
\begin{align*}
&\max \big(
\| D \mathcal{V}_d(\bar{\by})- D \mathcal{V}_d(\by_d) \|_{L^2(0,\infty;Y)},
\| D \mathcal{V}_d(\bar{\by})- D \mathcal{V}_d(\by_d) \|_{L^\infty(0,\infty;Y)}
\big) \\
&\qquad \leq M \| \bar{\by}- \by_d \|_{W_\infty}
\leq  M \| \by_0 \|_Y^d.
\end{align*}
The second estimate on $\bar{u}-u_d$ follows.
\end{proof}


\subsection{Estimates for the pressure}

It is well-known that for $\by_0 \in Y$, the pressure term that can be associated to the Navier-Stokes equations is a distribution only (see, e.g., \cite{Sim99}, \cite[Chapter III-\S3]{Tem79}).
In the following, we redemonstrate this fact and we argue that a result analogous to Theorem \ref{theo19} also holds for the pressure, provided the latter is considered in $W^{-1,\infty}(0,\infty; L_0^2(\Omega)) = W^{1,1}_0(0,\infty; L_0^2(\Omega))'$ with
\begin{equation*}
W^{1,1}_0(0,\infty; L_0^2(\Omega)) = \left\{v \in W^{1,1}(0,\infty;L_0^2(\Omega)) \ | \ v(0) = 0 \right\}
\end{equation*}
and
\begin{equation*}
L_0^2(\Omega) = \left\{v \in L^2(\Omega) \ | \ \int_\Omega{v(x)\, \mathrm{d}x = 0}\right\}.
\end{equation*}
We define similarly $W_0^{1,1}(0,\infty;\mathbb{H}_0^1(\Omega))$.
We recall here that $W^{1,1}_0(0,\infty; \mathbb{H}_0^1(\Omega))$ embeds continuously into $L^\infty(0,\infty;\mathbb{H}_0^1(\Omega)) \cap L^2(0,\infty;\mathbb{H}_0^1(\Omega))$. Further the elements $\bphi$ of $W^{1,1}_0(0,\infty; \mathbb{H}_0^1(\Omega))$ can be identified a.e.\@ with continuous functions on $[0,\infty)$ and satisfy $\lim_{t\to \infty} \| \bphi(t)\|_{\mathbb{H}_0^1(\Omega)}=0$. We use the properties of Banach-space valued functions as summarized in \cite[Chapter II-\S 5]{Boy13}.

\begin{lemma} \label{lemma:pressure}
Let $(\by,u) \in W_\infty \times L^2(0,\infty;U)$ be such that $\dot{\by}= A \by - F(\by) + Bu$. Then, there exists a unique $p \in W^{-1,\infty}(0,\infty;L_0^2(\Omega))$ such that
\begin{equation*}
\dot{\by}= A \by - F(\by) + Bu - \nabla p \quad
\text{in $W^{1,1}(0,\infty;\mathbb{H}_0^1(\Omega))'$},
\end{equation*}
that is,
\begin{align}
-\int_0^\infty \langle \by(t), \dot{\bphi}(t) \rangle_Y \, \mathrm{d} t
= & \ \int_0^\infty \big\langle A \by(t) - F(\by(t)) + Bu(t), \bphi(t) \big\rangle_{\mathbb{H}^{-1}(\Omega),\mathbb{H}_0^1(\Omega)} \, \mathrm{d} t \notag \\
& \qquad + \langle p, \divv \bphi \rangle_{W^{-1,\infty}(0,\infty,L_0^2(\Omega)),W_0^{1,1}(0,\infty;L_0^2(\Omega))}, \label{eq:weakNS}
\end{align}
for all $\bphi \in W_0^{1,1}(0,\infty;\mathbb{H}_0^1(\Omega))$.
Moreover,
\begin{equation} \label{eq:estimate_pi}
\| p \|_{W^{-1,\infty}(0,\infty,L_0^2(\Omega))}
\leq M \big( \| \by \|_{W_\infty} + \| \by \|_{W_\infty}^2 + \| u \|_{L^2(0,\infty;U)} \big),
\end{equation}
for a constant $M$ independent of $(\by,u)$.
\end{lemma}

\begin{proof}
We follow the technique consisting in integrating the state equation, see, e.g., \cite[Chapter V-\S 1]{Boy13} and introduce
\begin{equation}\label{eq:103}
\bG(t)= \by(t) {- \by_0} + \int_0^t \bg(s) \, \mathrm{d} s, \quad \text{with: }
\bg(s)= A \by(s) - F(\by(s)) + Bu(s).
\end{equation}
It can be easily shown that $\bg \in L^2(0,\infty;\mathbb{H}^{-1}(\Omega))$ and that there exists a constant $M>0$ independent of $(\by,u)$ such that
\begin{equation} \label{estimate:g}
\| \bg \|_{L^2(0,\infty;\mathbb{H}^{-1}(\Omega))}
\leq M \big( \| \by \|_{W_\infty} + \| \by \|_{W_\infty}^2 + \| u \|_{L^2(0,\infty;U)} \big).
\end{equation}
This estimate can be obtained with the Cauchy-Schwarz inequality and Proposition \ref{prop:estimates_oseen}(i), which also holds true in $\mathbb{H}^{-1}(\Omega)$ (in place of $V'$). Since $\by \in W_\infty$, it further follows that $\bG$ is a continuous function of time with values in $\mathbb{H}^{-1}(\Omega)$.
Moreover, $\langle \bG(t),\boldsymbol{\psi} \rangle_{\mathbb{H}^{-1}(\Omega),\mathbb{H}_0^1(\Omega)} = 0$ for all $t \in [0,\infty)$ and $\boldsymbol{\psi} \in V$. Hence for all $t \in [0, \infty)$, there exists a unique $\mathcal{P}(t) \in L_0^2(\Omega)$ such that $\bG(t) = -\nabla \mathcal{P}(t)$, see, e.g., \cite[Theorem IV.2.3]{Boy13}. Let us prove that $\mathcal{P} \in C([0,\infty), L^2_0(\Omega))$. Recall first that there exists an operator $\mathcal{K} \in \mathcal{L}(L_0^2(\Omega),\mathbb{H}_0^1(\Omega))$ with the property that
\begin{equation*}
\divv(\mathcal{K} \rho) = \rho, \quad \forall \rho \in L_0^2(\Omega),
\end{equation*}
see \cite[Theorem IV.3.1]{Boy13}. Let $\rho \in L_0^2(\Omega)$ be arbitrary and let $\bphi= \mathcal{K} \rho$. For all $t$ and $\tau$ in $[0, \infty)$, we have
\begin{align*}
\langle\mathcal{P}(t) - \mathcal{P}(\tau), \rho \rangle_{L_0^2(\Omega)}
& = -\langle \nabla \mathcal{P}(t) - \nabla \mathcal{P}(\tau), \bphi \rangle_{\mathbb{H}^{-1}(\Omega),\mathbb{H}_0^1(\Omega)} \\
&= \langle \bG(t) - \bG(\tau), \bphi \rangle_{\mathbb{H}^{-1}(\Omega),\mathbb{H}_0^1(\Omega)} \\
& \leq \| \mathcal{K} \|_{\mathcal{L}(L_0^2(\Omega),\mathbb{H}_0^1(\Omega))} \| \bG(t) - \bG(\tau) \|_{\mathbb{H}^{-1}(\Omega)} \| \rho \|_{L_0^2(\Omega)}.
\end{align*}
It follows that $\| \mathcal{P}(t)-\mathcal{P}(\tau) \|_{L_0^2(\Omega)} \leq M \| \bG(t)-\bG(\tau) \|_{\mathbb{H}^{-1}(\Omega)}$, which concludes the proof of continuity of $\mathcal{P}$.
We now introduce the distributional derivative $p= \frac{\mathrm{d}}{\mathrm{d} t} \mathcal{P}$ and establish that $p \in W^{-1,\infty}(0,\infty;L_0^2(\Omega))$. Let $\rho \in \mathcal{C}_c^\infty(0,\infty;L_0^2(\Omega))$ be arbitrary and set ${\bphi}(t) = \mathcal{K} \rho(t)$. Note that ${\bphi} \in  \mathcal{C}_c^\infty(0,\infty;\mathbb{H}^{1}_0(\Omega))$. We have
\begin{align*}
\langle p, \rho \rangle
= \ & - \int_0^\infty \langle \mathcal{P}(t), \dot{\rho}(t) \rangle_{L_0^2(\Omega)} \, \mathrm{d} t
= - \int_0^\infty \langle \mathcal{P}(t), \divv \dot{\bphi}(t) \rangle_{L_0^2(\Omega)} \, \mathrm{d} t \\
= \ & \int_0^\infty \langle \nabla \mathcal{P}(t), \dot{\bphi}(t) \rangle_{\mathbb{H}^{-1}(\Omega),\mathbb{H}_0^1(\Omega)} \, \mathrm{d} t \\
= \ & - \int_0^\infty \langle \by(t) {- \by_0} - {\textstyle \int_0^t } \bg(s)\, \mathrm{d} s, \dot{\bphi}(t) \rangle_{\mathbb{H}^{-1}(\Omega),\mathbb{H}_0^1(\Omega)} \, \mathrm{d} t \\
= \ & -\int_0^\infty \langle \by(t) {- \by_0}, \dot{\bphi}(t) \rangle_{\mathbb{H}^{-1}(\Omega),\mathbb{H}_0^1(\Omega)}
+ \langle \bg(t), \bphi(t) \rangle_{\mathbb{H}^{-1}(\Omega),\mathbb{H}_0^1(\Omega)} \, \mathrm{d} t.
\end{align*}
Recalling the embedding of $W_0^{1,1}(0,\infty;\mathbb{H}_0^1(\Omega))$ in $L^2(0,\infty;\mathbb{H}_0^1(\Omega))$, we deduce that
\begin{equation*}
\langle p, \rho \rangle
\leq M \big( \| \by \|_{L^\infty(0,\infty;\mathbb{H}^{-1}(\Omega))} + \| \bg \|_{L^2(0,\infty;\mathbb{H}^{-1}(\Omega))} \big) \| \bphi \|_{W_0^{1,1}(0,\infty;\mathbb{H}_0^1(\Omega))}.
\end{equation*}
Using then estimate \eqref{estimate:g}, we obtain that $p$ can be extended to an element of $W^{-1,\infty}(0,\infty;L_0^2(\Omega))$ satisfying estimate \eqref{eq:estimate_pi}.

With the same calculations as above, we can show that for all $\bphi \in W_0^{1,1}(0,\infty;\mathbb{H}_0^1(\Omega))$,
\begin{equation*}
\langle p, \divv \bphi \rangle=
- \int_0^\infty \langle \by(t), \dot{\bphi}(t) \rangle_{\mathbb{H}^{-1}(\Omega),\mathbb{H}_0^1(\Omega)}
+ \langle \bg(t), \bphi(t) \rangle_{\mathbb{H}^{-1}(\Omega),\mathbb{H}_0^1(\Omega)} \, \mathrm{d} t,
\end{equation*}
which proves that $p$ satisfies \eqref{eq:weakNS}. Let us prove the uniqueness. Let $\tilde{p} \in W^{-1,\infty}(0,\infty;L_0^2(\Omega))$ satisfy \eqref{eq:weakNS}. Let $\rho \in W_0^{1,1}(0,\infty;L_0^2(\Omega))$ be arbitrary and let us set $\bphi= \mathcal{K} \rho$. Then, by \eqref{eq:weakNS}, we have
\begin{equation*}
\begin{aligned}
0 &= \langle p- \tilde{p}, \divv \bphi \rangle_{W^{-1,\infty}(0,\infty;L_0^2(\Omega)),W_0^{1,1}(0,\infty;L_0^2(\Omega))} \\
&= \langle p- \tilde{p}, \rho \rangle_{W^{-1,\infty}(0,\infty;L_0^2(\Omega),W_0^{1,1}(0,\infty;L_0^2(\Omega))},
\end{aligned}
\end{equation*}
which proves that $p= \tilde{p}$ and concludes the proof.
\end{proof}

We have the following result, extending Theorem \ref{theo19}.

\begin{proposition}\label{prop:pressure}

Let $d \geq 2$. There exists $M > 0$ such that for all $\mathbf{y}_0 \in Y$ with $\|\mathbf{y}_0 \|_Y \leq \delta_6$,
\begin{equation*}
\| \bar{p} - p_d \|_{W^{-1,\infty}(L_0^2)} \leq M \| \mathbf{y}_0 \|_Y^d,
\end{equation*}
where $\bar{p}$ and $p_d$ denote the pressure terms associated with $(\bar{\by},\bar{u})$ and $(\by_d,u_d)$ respectively.
\end{proposition}

\begin{proof}
We have introduced in the proof of Lemma \ref{lemma:pressure} the term $\bg$ associated with a feasible pair $(\by,u)$. Let us denote by $\bar{\bg}$ and $\bg_d$ the corresponding terms associated with $(\bar{\by},\bar{u})$ and $(\by_d,u_d)$. One can verify that as a consequence of Theorem \ref{theo19},
$\| \bar{\bg} - \bg_d \|_{L^2(0,\infty;\mathbb{H}^{-1}(\Omega))}
\leq M \| \by_0 \|_Y^d$.
Proposition \ref{prop:pressure} follows then with similar calculations to those performed in the proof of Lemma \ref{lemma:pressure}.
\end{proof}

\section{A numerical example}

In this section, we present numerical simulations for the two-dimensional Navier-Stokes equations and computed feedback laws of order 2 and 3. The discretization procedure and the example setups are classical and are taken from \cite{BehBH17}. The main purpose is to show that the computation of higher order feedback  laws is possible and, depending on the chosen parameters, visible differences to a Riccati-based feedback law can be observed.

\subsection{Setup and discretization}

We briefly summarize the numerical implementation provided in \cite{BehBH17}. Therein a \emph{Taylor-Hood} $P_{2}$-$P_{1}$ finite element discretization for a two dimensional wake behind a cylinder is discussed. The computational domain $\Omega=(0,2.2)\times (0,0.41)$ as well as a non uniform grid are shown in Figure \ref{fig:cylinder_grid}. For all simulations, we use the \emph{Reynolds number} $\mathrm{Re}:=\frac{1}{\nu}=90$ and the parabolic inflow profile discussed in \cite{BehBH17}.
\begin{figure}
 \begin{center}
	 \includegraphics[scale=1.6]{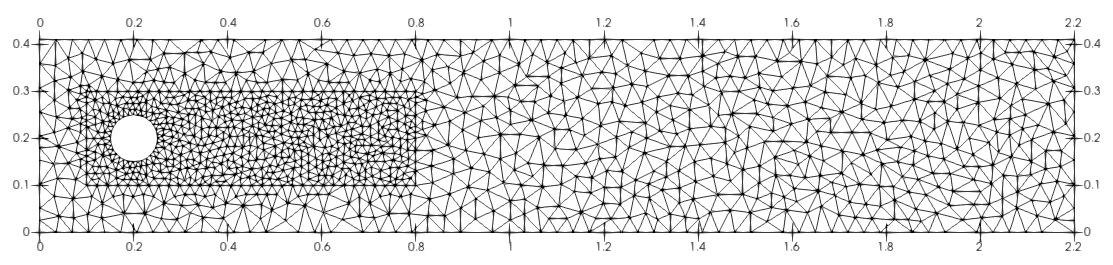}
 \end{center}	
  \caption{Geometry and non uniform grid.}
 \label{fig:cylinder_grid}
\end{figure}
For the upper and lower end of the geometry, \emph{no slip boundary conditions} are employed. The outflow is modeled by \emph{do nothing boundary conditions} on the right end of the geometry. For the desired stabilization, we utilize a distributed, separable control acting in the control domain $\Omega_{c}:= [0.27,0.32]\times [0.15,0.25]$. In particular, the control operator is of the form
\begin{align*}
Bu=\sum_{\ell=1}^3	\begin{bmatrix} 0 \\ w_{\ell}(x_2) \end{bmatrix} u_{\ell}(t) + \begin{bmatrix} w_{\ell}(x_2) \\ 0 \end{bmatrix} u_{\ell+3}(t) ,
\end{align*}
where the control shape functions $w_{1},w_{2}$ and $w_{3}$ are piecewise linear functions which are constant along the $x_1$-direction.

 The finite element discretization is computed in \emph{FEniCS} and the resulting matrices associated with the spatial semidiscretization are exported to \mbox{MATLAB}. As described in detail in \cite{BehBH17}, the (spatially) discrete system takes the form
 \begin{equation}\label{eq:discrete_notshifted}
\begin{aligned}
	E\dot{z}(t)&= -Kz(t) + H (z(t)\otimes z(t)) + Bu(t) + Gq(t)+ f_{z}, \\
	0&= G^{T}z(t) + f_{q},
\end{aligned}
\end{equation}
where $E,K \in \mathbb R^{n_v \times n_{v}}$ are the mass and stiffness matrices, $G^T \in \mathbb R^{n_p \times n_v}$ represents the discrete divergence operator, the tensor matricization $H \in \mathbb R^{n_{v}\times n_{v}^2}$ represents the trilinear form \eqref{eq:trilinear_form} and $B\in \mathbb R^{n_v \times 6}$ is the discrete control operator. Note that $H$ can be constructed in such a way that $H (z_{1}\otimes z_{2})=H(z_{2}\otimes z_{1})$ for any $z_{1},z_{2} \in \mathbb R^{n_{v}}$. The time invariant vectors $f_{z} \in \mathbb R^{n_v}$ and $f_{q}\in \mathbb R^{n_p}$ are due to the elimination of the boundary nodes. The following results correspond to a discretization level with $n_{v}=9356$ and $n_{p}=1289$. The velocity profile of the unstable steady state solution $\bar{z}$ shown in Figure \ref{fig:steady_flow} is obtained by a Picard iteration applied to the uncontrolled stationary system, i.e., system \eqref{eq:discrete_notshifted} with $\dot{z}(t)=0$ and $u(t)=0$.
To illustrate that the
controller stabilizes this steady state solution, we start the transient simulations of the closed-loop systems from the slightly randomly perturbed
steady state $z(0)=\bar{z}+\frac{\|\bar{z}\|_2}{2000}\cdot \texttt{randn}(n_v,1)$.

\begin{figure}[htb]
 \begin{center}
	 \includegraphics[scale=1.8]{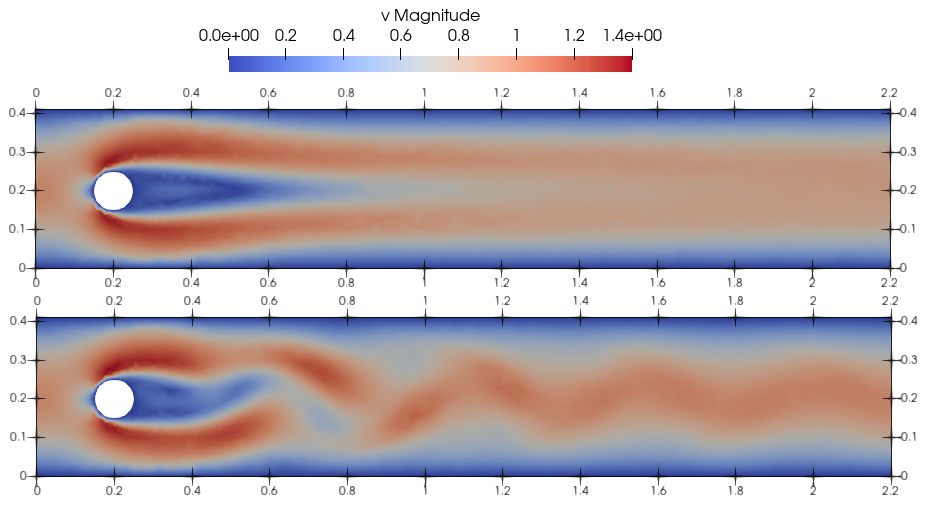}
 \end{center}	
 \caption{The steady state and a snapshot of the transient flow regime.}
 \label{fig:steady_flow}
\end{figure}

\subsection{Reformulation as an ODE system}

System \eqref{eq:discrete_notshifted} is a system of differential-algebraic equations (DAEs) and hence the results from above are not readily applicable. While a thorough analysis in the framework of control of DAEs is certainly of interest, at this point we employ a reformulation initially proposed in \cite{HeiSS08} that allows to rewrite the dynamics as a set of ODEs for the velocity vector $z$. As in \eqref{eq:gen_hom_NSE}, we consider the shifted variables $y=z-\bar{z}$ and $p=q-\bar{q}$, respectively. Consequently, we obtain
 \begin{equation}\label{eq:discrete_shifted}
\begin{aligned}
	E\dot{y}(t)&= Ay(t) + H (y(t)\otimes y(t)) + Bu(t) + Gp(t), \\
	0&= G^{T}y(t),
\end{aligned}
\end{equation}
where $A=-K+H(\bar{z}\otimes I+I\otimes \bar{z})$. Let us note that the second equation implies $G^{T}\dot{y}(t)=0$. Following \cite[Section 3]{HeiSS08}, from the first equation, we thus obtain
\begin{align*}
0=G^{T}\dot{y}(t) = G^{T}E^{-1}\left( A y(t) + H (y(t)\otimes y(t)) + Bu(t) + Gp(t) \right).
\end{align*}
We can now eliminate the pressure from \eqref{eq:discrete_shifted} using  the relation
\begin{align*}
	 p(t)= - (G^{T}E^{-1}G)^{-1}G^{T}E^{-1} \left( Ay (t) + H(y(t)\otimes y(t)) + Bu(t)\right).
\end{align*}
With the notation $P=I- G(G^{T}E^{-1}G)^{-1}G^{T}E^{-1}$ this yields the system
\begin{align*}
	E \dot{y}(t) = PAy(t) + P H(y(t)\otimes y(t)) + PBu(t).
\end{align*}
In fact, as has been discussed in \cite{BenSSW13}, the matrix $P=P^{2}$ as a discrete realization of the \emph{Leray projector}. Since $G^Ty=0$, we have  $P^{T}y(t)=y(t)$ so that we can multiply the last equation by $P$ to obtain
\begin{align*}
  (PEP^{T}) \dot{y}(t) = (PAP^{T})y(t) + \left(PHP^{T}\otimes P^{T}\right)(y(t)\otimes y(t)) + (PB) u(t).
\end{align*}
Finally, by means of a decomposition $P=\Theta_{\ell}\Theta_{r}^{T}$ with $\Theta_{\ell}^{T} \Theta_{r}=I$ we can project onto the $n_{v}-n_{p}$ dimensional subspace $\mathrm{range}(P)$ and arrive at the ODE system
\begin{align}\label{eq:disc_divergence_free}
\underbrace{(\Theta_{r}^{T}E\Theta_{r})}_{\widetilde{E}} \dot{\tilde{y}}(t) = \underbrace{(\Theta_{r}^{T }A \Theta_{r})}_{\widetilde{A}} \tilde{y}(t) +\underbrace{(\Theta_{r}^{T} H \Theta_{r}\otimes \Theta_{r})}_{\widetilde{H}} \tilde{y}(t)\otimes \tilde{y}(t) + \underbrace{(\Theta_{r}^{T} B)}_{\widetilde{B}} u(t),
\end{align}
where $\tilde{y}=\Theta_{\ell}^{T} y(t)$. For the initialization, we use $\tilde{y}(0)=\Theta_{\ell}^{T} y_{0}$. At this point, we emphasize that the explicit formulas yield dense matrices and thus are rather a theoretical tool. In particular, an explicit computation of $\widetilde{H}$ is infeasible for the problem dimension considered here. As a remedy, we work with an implementation that applies the above operations whenever a matrix vector multiplication is needed.

\subsection{Computing the feedback gain}

With the previous considerations in mind, we focus on the stabilization problem
\begin{equation} \label{eq:discrete_NLQprob}
\inf_{\begin{subarray}{c}   u \in L^2(0,\infty;\mathbb R^{6}) \end{subarray}} J(\tilde{y}_{0},u), \quad \text{subject to: } e(\tilde{y}_{u},u)= (0,\tilde{y}_0)
\end{equation}
where
\begin{align*}
J(\tilde{y}_{u},u)= & \frac{1}{2} \int_0^\infty \| \Theta_{r}\tilde{y}_{u}(t) \|^2_{\mathbb R^{n_{v}}} \dd t + \frac{\alpha}{2} \int_0^\infty \| u(t) \|_{\mathbb R^{6}}^2 \dd t \\
e(\tilde{y}_{u},u)= & \big( \widetilde{E}\dot{\tilde{y}}_{u}-(\widetilde{A}\tilde{y}_{u} + \widetilde{H}(\tilde{y}_{u}\otimes \tilde{y}_{u}) + \widetilde{B}u), \tilde{y}(0) \big).
\end{align*}
We illustrate the effect of higher order feedback laws by computing the first two non trivial derivatives $D^{2}\mV(0)$ and $D^{3}\mV(0)$, respectively. For the computation of $D^{2}\mV(0)\equiv \Pi\in \mathbb R^{(n_{v}-{n_p})\times (n_{v}-n_{p})}$, we have to solve the algebraic matrix Riccati equation
\begin{align*}
	\widetilde{A}^{T} \Pi \widetilde{E} + \widetilde{E}^{T}\Pi \widetilde{A} - \widetilde{E}^{T} \Pi \widetilde{B} \widetilde{B}^{T} \Pi \widetilde{E} + \Theta_{r}^{T} \Theta_{r} = 0,
\end{align*}
which in our case was done by means of the \mbox{MATLAB} function $\texttt{care}$. For the third order tensor $D^{3}\mV(0)\equiv \mathcal{X} \in \mathbb R^{(n_{v}-n_{p})^{3}}$ we have to solve a linear system of the form  $\mathcal{A}^{T}\mathcal{X} = \mathcal{F}$
where
\begin{equation}\label{eq:numerics_aux1}
\begin{aligned}
\mathcal{A}&= \widetilde{E}\otimes \widetilde{E} \otimes \widetilde{A}_{\pi}	 + \widetilde{E} \otimes \widetilde{A}_{\pi}\otimes \widetilde{E}+ \widetilde{A}_{\pi}\otimes \widetilde{E} \otimes \widetilde{E}, \ \ \ \widetilde{A}_{\pi} = \widetilde{A}-\frac{1}{\alpha}\widetilde{B}\widetilde{B}^{T} \Pi \widetilde{E}, \\
\mathcal{F}&=-2\left(\widetilde{H}^{T}\otimes \widetilde{E}^{T} + \widetilde{E}^{T} \otimes \widetilde{H}^{T}+ (I\otimes \mathcal{P}^{T})(\widetilde{H}^{T}\otimes \widetilde{E}^{T}) \right) \pi,
\end{aligned}
\end{equation}
where $\pi=\mathrm{vec}(\Pi)$ denotes the vectorization of $\Pi$ and the permutation matrix $\mathcal{P}$ is given by
\begin{align*}
  \mathcal{P}= \begin{bmatrix} I\otimes  e_{1},\dots,I\otimes  e_{n_v-n_p} \end{bmatrix} \in \mathbb R^{(n_v-n_p)^{2} \times (n_v-n_p)^{2} }.
\end{align*}
Let us emphasize that $\mathcal{F}$ is the discrete realization of the term $\mathcal{R}_{3}$ in \eqref{eq:Lyapunov1}. In particular, the tensor $\mathcal{F}$ is symmetric. Note that computing a solution $\mathcal{X}$ to $\mathcal{A}^{T}\mathcal{X}=\mathcal{F}$ is infeasible without using further tools such as model order reduction or tensor calculus as storing the vector $\mathcal{X}\in \mathbb R^{(n_v-n_p)^{3}}$ already requires more than $4$ TB of data. As a remedy, we aim for a direct computation of the corresponding feedback gain
\begin{align}\label{eq:numerics_aux2}
\widetilde{K} = (\widetilde{E}^{T} \otimes \widetilde{E}^{T} \otimes \widetilde{B}^{T})\mathcal{X}
\end{align}
without explicitly computing $\mathcal{X}$. With this in mind, we proceed as in \cite{BreKP17c} and utilize a quadrature-based approximation that has been analyzed in \cite{Gra04}. From \cite[Lemma 3]{Gra04}, it follows that
\begin{align*}
	\mathcal{A}^{-1}  = -\int_{0}^{\infty}  \left( e^{t \widetilde{E}^{-1} \widetilde{A}_{\pi}} \widetilde{E}^{-1} \right) \otimes
	 \left( e^{t \widetilde{E}^{-1} \widetilde{A}_{\pi}} \widetilde{E}^{-1} \right)
	 \otimes  \left( e^{t \widetilde{E}^{-1} \widetilde{A}_{\pi}} \widetilde{E}^{-1} \right) \,\mathrm{d}t.
\end{align*}
As shown in \cite[Theorem 9]{Gra04}, the previous integral can be well approximated by a tensor sum of the form
\begin{align}\label{eq:numerics_aux3}
\mathcal{A}^{-1} \approx - \sum_{j=-r}^{r} \frac{2 w_{j}}{\lambda} \left( e^{\frac{t_{j}}{\lambda} \widetilde{E}^{-1} \widetilde{A}_{\pi}} \widetilde{E}^{-1} \right)
\otimes \left( e^{\frac{t_{j}}{\lambda} \widetilde{E}^{-1} \widetilde{A}_{\pi}} \widetilde{E}^{-1} \right) \otimes \left( e^{\frac{t_{j}}{\lambda} \widetilde{E}^{-1} \widetilde{A}_{\pi}} \widetilde{E}^{-1} \right)
\end{align}
where $t_{j}$ and $w_{j}$ are suitable quadrature points and weights and $\lambda$ denotes a constant determined by the spectrum of the matrix pencil $(\widetilde{E},\widetilde{A})$. Combining the representation in \eqref{eq:numerics_aux1}, \eqref{eq:numerics_aux2} and \eqref{eq:numerics_aux3}, we obtain the following approximation formula for the feedback gain
\begin{align*}
	\widetilde{K} &=-\sum_{j=-r}^{r} \frac{2w_{j}}{\lambda}   \left(    (e^{\frac{t_{j}}{\lambda} \widetilde{E}^{-1} \widetilde{A}_{\pi}})^{T}\right)
\otimes \left( (e^{\frac{t_{j}}{\lambda} \widetilde{E}^{-1} \widetilde{A}_{\pi}})^{T} \right) \otimes \left( \widetilde{B}^{T} \widetilde{E}^{-T} (e^{\frac{t_{j}}{\lambda} \widetilde{E}^{-1} \widetilde{A}_{\pi}})^{T} \right)\mathcal{F} \\[1ex]
&= \sum_{j=-r}^{r} \frac{4w_{j}}{\lambda}   \left(    (e^{\frac{t_{j}}{\lambda} \widetilde{E}^{-1} \widetilde{A}_{\pi}})^{T}\right)
\otimes \left( (e^{\frac{t_{j}}{\lambda} \widetilde{E}^{-1} \widetilde{A}_{\pi}})^{T} \right) \otimes \left( \widetilde{B}^{T} \widetilde{E}^{-T} (e^{\frac{t_{j}}{\lambda} \widetilde{E}^{-1} \widetilde{A}_{\pi}})^{T} \right)
\\ &\qquad  \times
\left(\widetilde{H}^{T}\otimes \widetilde{E}^{T} + \widetilde{E}^{T} \otimes \widetilde{H}^{T}+ (I\otimes \mathcal{P}^{T})(\widetilde{H}^{T}\otimes \widetilde{E}^{T}) \right) \pi,
\end{align*}
with $r=30$ in the numerical examples.
By use of algebraic manipulations such as reshaping and transposition of matrices, the computation of the permutation matrix $\mathcal{P}$ as well as computation of the dense matricization $\widetilde{H}$ can be avoided. As a consequence, we obtain an approximation of $\widetilde{K}\in \mathbb R^{6 (n_{v}-n_{p})^{2}}$ whose storage requires less than 4 GB of data. Let us point out that the above considerations do not fully break the curse of dimensionality but nevertheless allow us to compute a third order feedback law even for a spatially discretized PDE. For the simulation of the time-varying systems, we make use of the MATLAB function \texttt{ode23} with the standard relative error tolerance $10^{-3}$. In each time step, the control laws $u_2(\tilde{y})$ and $u_{3}(\tilde{y})$ are obtained via
\begin{align*}
	u_2(\tilde{y}) &=  -\frac{1}{\alpha} \widetilde{B}^T\Pi \widetilde{E}\tilde{y}, \\
	u_{3}(\tilde{y}) &= -\frac{1}{\alpha} \widetilde{B}^T\Pi \widetilde{E}\tilde{y} -\frac{1}{\alpha} \left(I_{6} \otimes \tilde{y}^{T} \otimes \tilde{y}^{T}\right) \widetilde{K},
\end{align*}
where $I_{6}$ denotes the identity matrix for the control space $\mathbb R^{6}$.

\subsection{Results}

Below, we present a numerical comparison for two different values of $\alpha$. In Figure \ref{fig:control_big_alpha}, the control laws corresponding to \eqref{eq:discrete_NLQprob} with $\alpha=1$ are shown. We observe that both feedback laws $u_{2}$ and $u_{3}$, respectively, exhibit a similar behavior and create  vortices which induce the desired control. Indeed, the control velocities in $x_{1}$-direction are of opposite sign (with the centered velocitiy field being negligible) while the control velocities in $x_{2}$-direction all have the same sign.

\begin{figure}[ht]
 	 \includegraphics[scale=0.5]{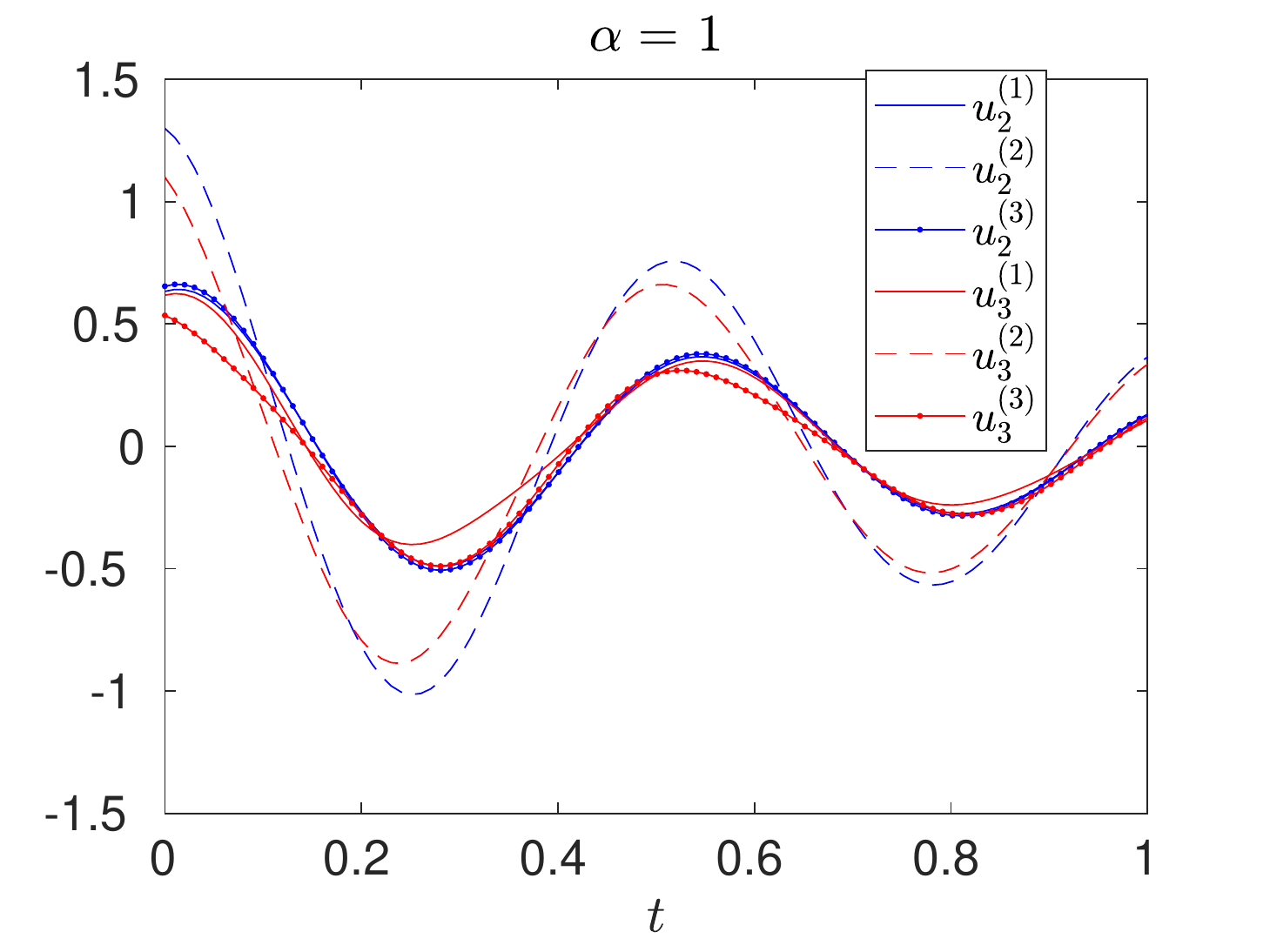}
	 \includegraphics[scale=0.5]{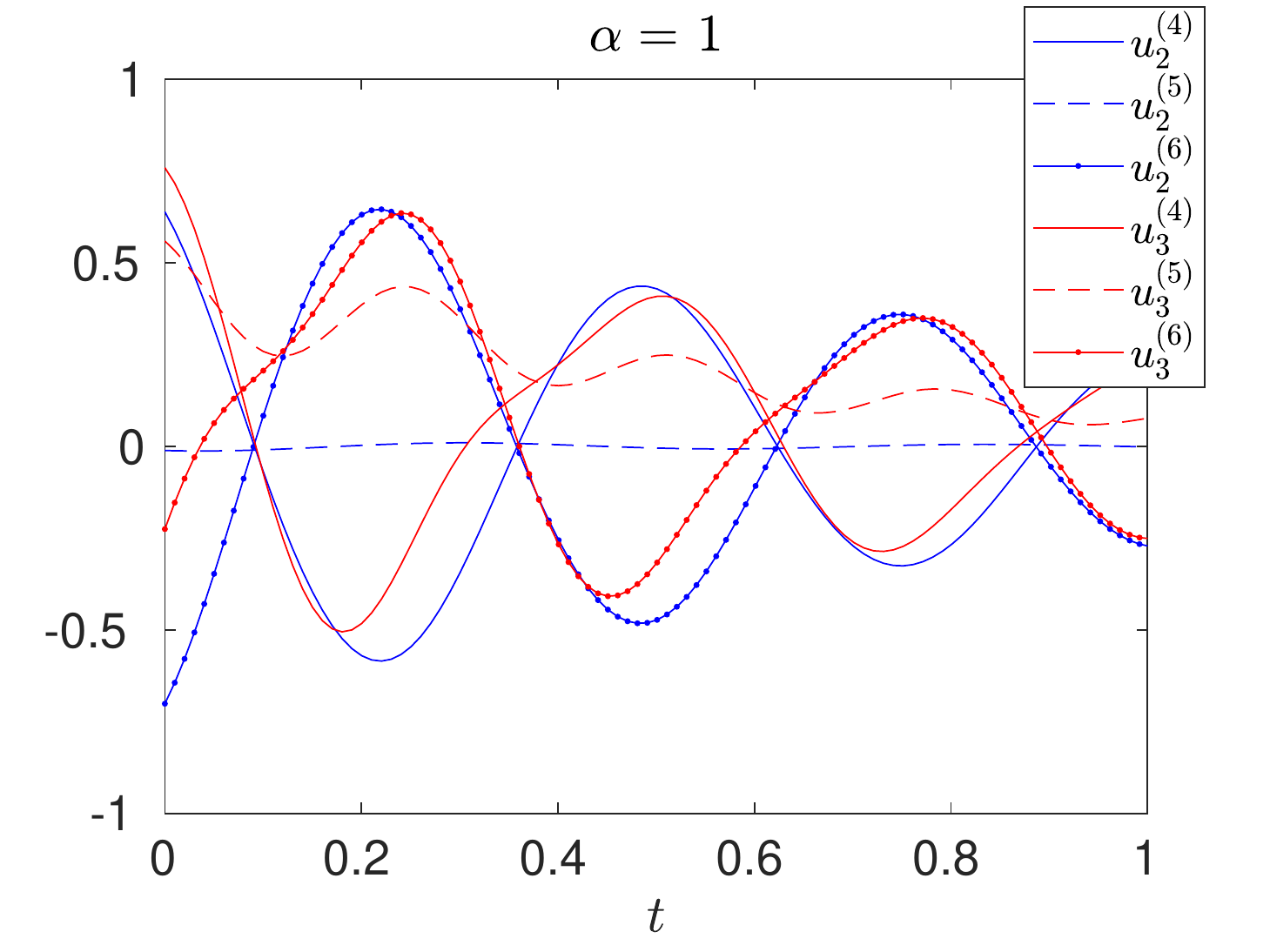}
 \caption{Control laws in $x_{2}$ (left) and $x_{1}$-direction (right) for $\alpha=1$.}
  \label{fig:control_big_alpha}
\end{figure}

For $\alpha=10^{-4}$, Figure \ref{fig:control_small_alpha} shows more visible differences between the control laws.

\begin{figure}[ht]
 	 \includegraphics[scale=0.5]{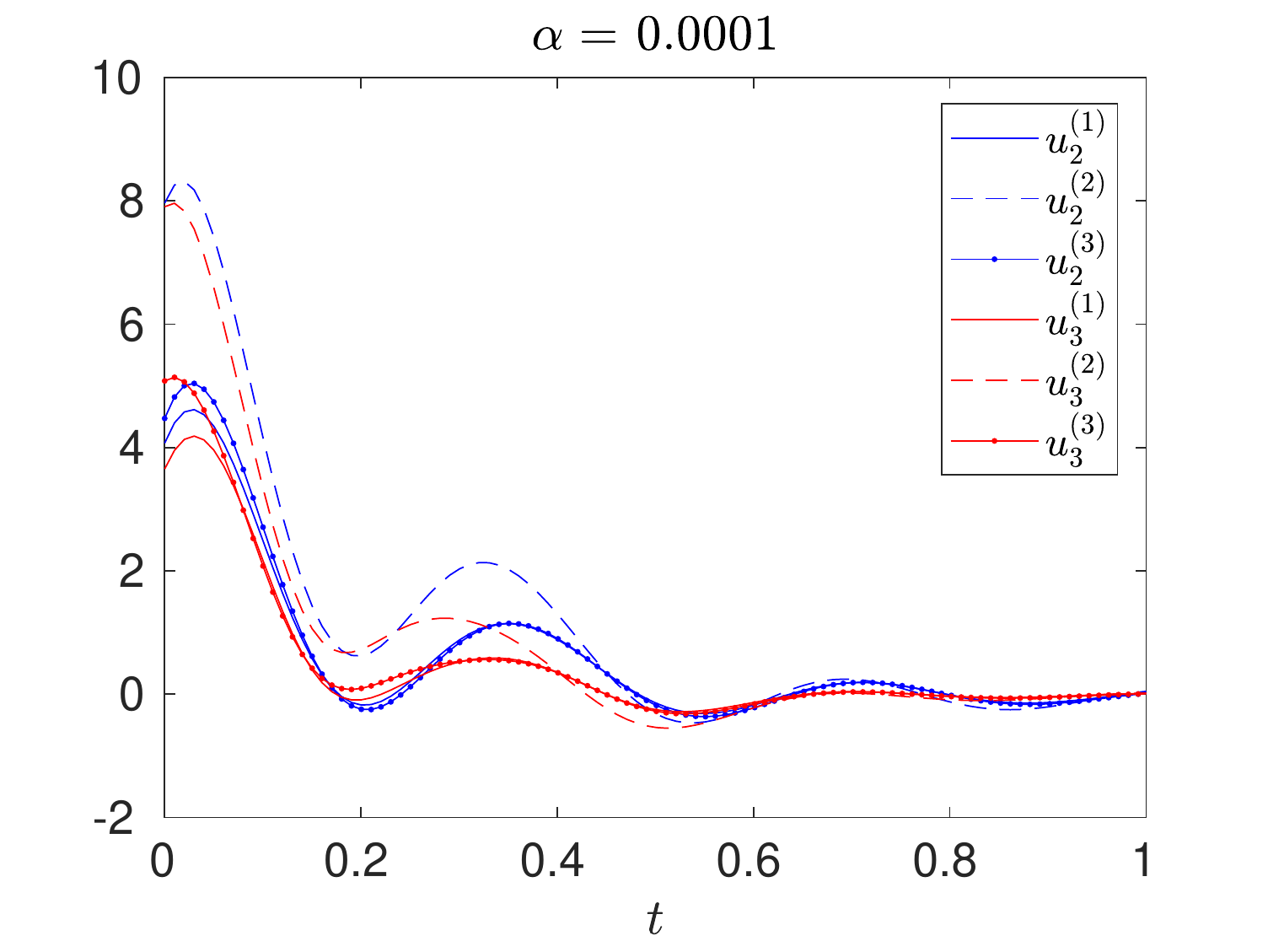}
	 \includegraphics[scale=0.5]{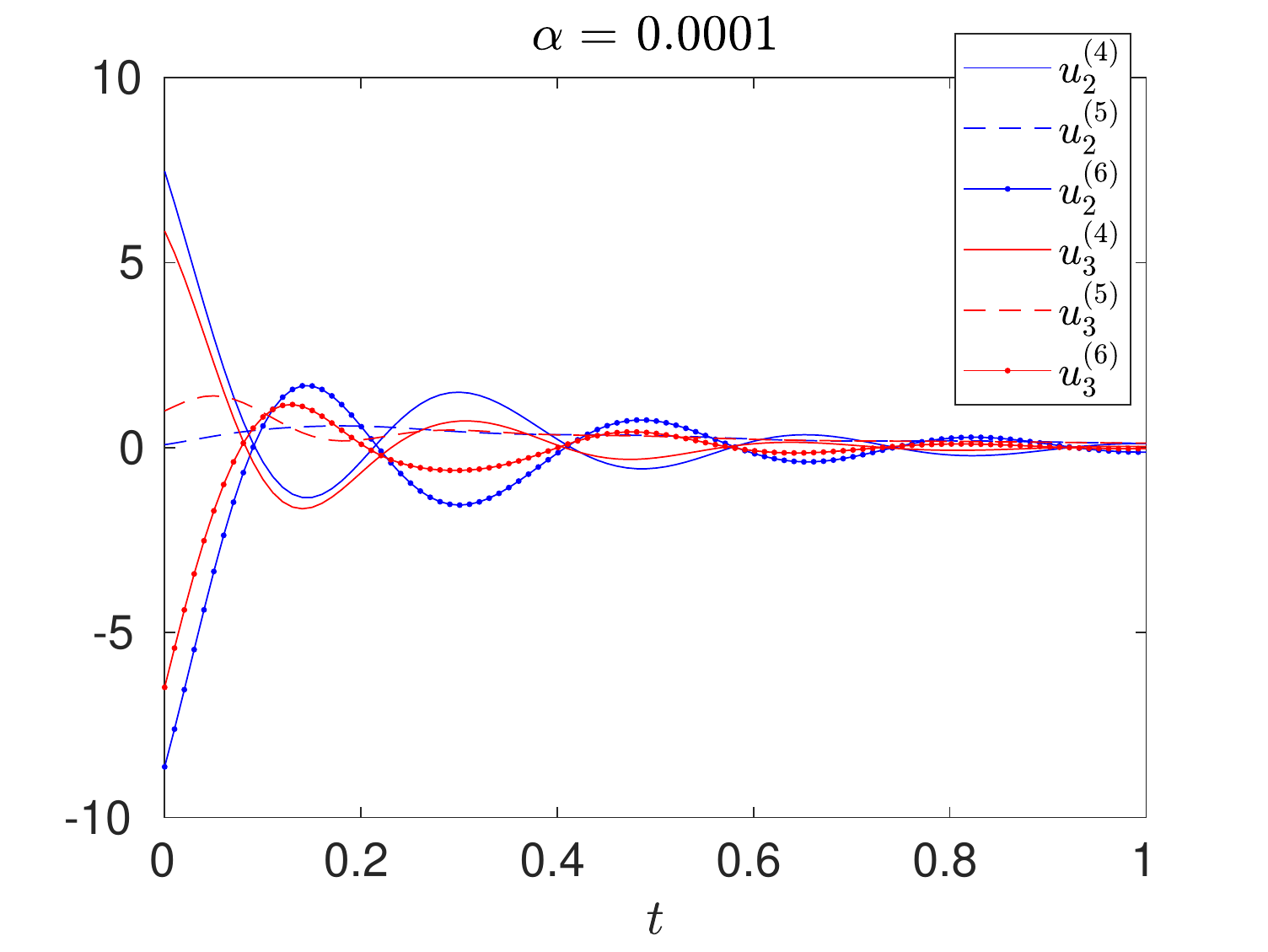}
 \caption{Control laws in $x_{2}$ (left) and $x_{1}$-direction (right) for $\alpha=10^{-4}$.}
 \label{fig:control_small_alpha}
\end{figure}

It would certainly be of interest to investigate the numerical convergence behavior as the order of the control laws increases. At the moment, however this is out of reach, and could be based on model reduction techniques in an independent numerical endeavor.
In Figure \ref{fig:control_small_alpha}, we observe that the amplitudes of the  $u_{3}$ controls decay more rapidly than those of the $u_{2}$ controls.
This is consistent with Figure  \ref{fig:control_norm}, where we compare the dynamical behavior of $\|u_{2}\|_{2}^{2}$ and $\|u_{3}\|_{2}^{2}$. Let us emphasize that for $\alpha=10^{-4}$, for all $t$, the norm of the control law $u_{3}(t)$ is smaller than the one of $u_{2}(t)$. For the values of the cost functionals, we obtain
\begin{align*}
J(\tilde{y}_{u_2},u_2)&=0.9546, && \hspace{-3cm} J(\tilde{y}_{u_3},u_3)=0.8432, \quad \text{for } \alpha=1, \\
J(\tilde{y}_{u_2},u_2)&=0.0128, && \hspace{-3cm}J(\tilde{y}_{u_3},u_3)=0.0125, \quad \text{for } \alpha=10^{-4},
\end{align*}
which  indicates that higher order feedback laws can be of interest for feedback stabilization.


\begin{figure}[ht]
	\includegraphics[scale=0.5]{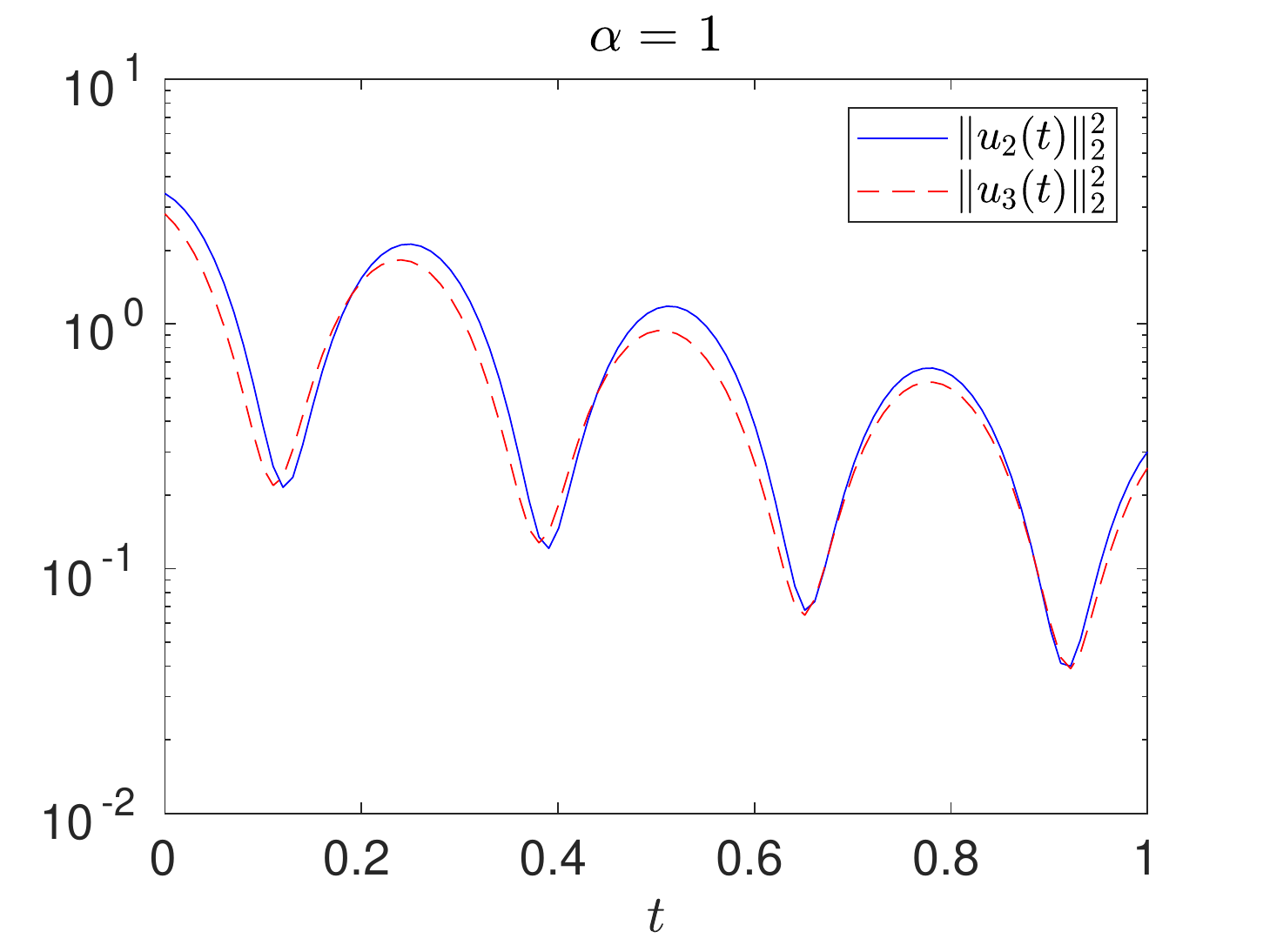}
	\includegraphics[scale=0.5]{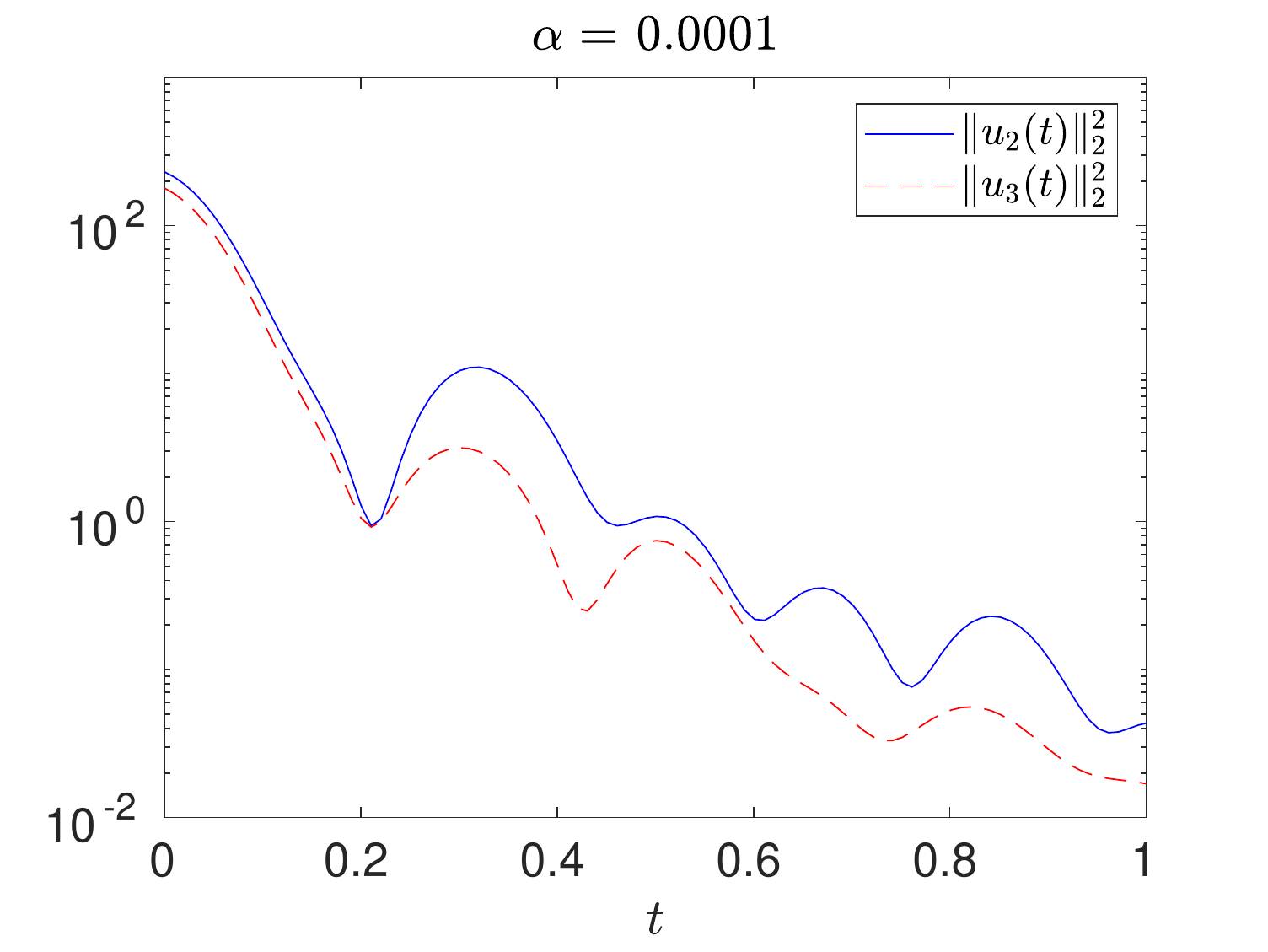}
	\caption{Dynamical behavior of the control norms for $\alpha=1$ and $\alpha=10^{-4}$.}
	\label{fig:control_norm}
\end{figure}

\section{Outlook}

In the present paper we demonstrated that the approach that we carried out for obtaining Taylor approximations to the value function of optimal control problems related to the Fokker-Planck equation, is also applicable for optimal control of the Navier-Stokes equations in dimension two. The question arises to which extent analogous results can be obtained for dimension three and for boundary control problems. In dimension three the situation will be significantly different from that of the current paper. It will not be possible to work with weak variational solutions. Rather one has to resort to strong variational solutions, and thus one can expect at best that the value function is smooth on $V$ rather than on $Y$. This leads to  difficulties for the operator representations of the derivatives of the value function. Alternatively one can start by analyzing \eqref{eq:Lyapunov1} as equations for abstract multilinear forms $D^k\mathcal{V}(0)$, which are not necessarily obtained of derivatives of $\mathcal{V}$. This is an approach which we plan to follow.

\subsection*{Acknowledgement}

This work was partly supported by the ERC advanced grant 668998 (OCLOC) under
the EU's H2020 research program.
The authors would like to thank Jan Heiland for making available his finite element based code for solving the state equation
as well as many helpful and interesting discussions on the
numerical examples.

\appendix

\section{Appendix}
The appendix is dedicated to the proof of Proposition \ref{proposition:more_reg}.

We follow the notation from, e.g., \cite{BT11} and define the following spaces
\begin{align*}
  V_n^s &:= \left\{ \by \in \bbH^s(\Omega)\, | \, \divv \by =0, \, \by \cdot \vec{n}=0 \text { on } \Gamma \right\}, \ s \ge 0, \\
  V_0^s&:= V_n^s , s \in [0,\frac{1}{2}), \\
 V_0^s&:= \left\{ \by \in \bbH^s(\Omega)\, | \, \divv \by =0, \, \by =0 \text{ on } \Gamma \right\}, \ s > \frac{1}{2}.
\end{align*}
Moreover, we consider $A_{\alpha} := A-\alpha I $
where $A$ is the Stokes-Oseen operator and $\alpha$ is such that $A_{\alpha}$ generates an exponentially stable and contractive semigroup on $Y$. From \cite[Theorem 20]{BT11}, let us recall that
\begin{align*}
\mD((-A_\alpha^*)^\theta) = [V_0^2,Y]_{1-\theta}= V_0^{2\theta}, \ \ \forall \theta \in  [0,1],\ \  \theta \neq \frac{1}{4}.
\end{align*}
While not needed for our purposes, let us emphasize that the case $\theta=\frac{1}{4}$ is also included in \cite[Theorem 20]{BT11}.

Using the above notation, for $\varepsilon \in (0,\frac{1}{2})$ let us consider the space
\begin{align*}
 W_\infty(V_0^{1+\varepsilon},(V_0^{1-\varepsilon})')&:= \left\{ \by \in L^2(0,\infty;V_0^{1+\varepsilon}) \, | \, \frac{\mathrm{d}}{\mathrm{d}t} \by \in L^2(0,\infty;(V_0^{1-\varepsilon})') \right\}
\end{align*}
As mentioned in, e.g., \cite{Bad12}, it holds that
\begin{align*}
[V_0^{1+\varepsilon},(V_0^{1-\varepsilon})']_{\frac{1}{2}} = [\mD((-A_\alpha^*)^{\frac{1+\varepsilon}{2}}),\mD((-A_\alpha^*)^{\frac{\varepsilon-1}{2}})]_{\frac{1}{2}}= \mD((-A_\alpha^*)^{\frac{\varepsilon}{2}}) = V_0^\varepsilon.
\end{align*}
From \cite[Theorem 4.2]{LioM72}, we thus conclude that
\begin{align*}
W_\infty(V_0^{1+\varepsilon},(V_0^{1-\varepsilon})')\hookrightarrow C_b([0,\infty);V_0^\varepsilon ),
\end{align*}
where $C_b$ denotes the space of continuous and bounded functions.


Before we continue, let us cite the following result from \cite{GruS91}.

\begin{proposition}(\cite[B.1]{GruS91})\label{prop:grub}
Let $\lambda,\mu,\omega \in \mathbb R$. One has for $f\in H^{\lambda + \mu}(\Omega)$ and $g\in H^{\lambda+\omega}(\Omega)$ (where $\Omega$ is a smooth open subset of $\mathbb R^n$):
\begin{align*}
\|fg \| _{H^{\lambda}(\Omega)} \le C \| f\| _{H^{\lambda+\mu}(\Omega)} \| g \| _{H^{\lambda+\omega}(\Omega)} ,
\end{align*}
\begin{itemize}
\item[(i)] when $\mu+\omega +\lambda \ge \frac{n}{2}$,
\item[(ii)] with $\mu \ge 0, \omega \ge 0, 2\lambda \ge -\mu -\omega ,$
\item[(iii)] except that $\mu+\omega + \lambda > \frac{n}{2}$ if equality holds somewhere in $(ii)$.

\end{itemize}
\end{proposition}

These estimates allow us to  bound the coupling terms appearing in the adjoint equation.

\begin{lemma}\label{lem:grubsol}
Let $\varepsilon \in (0,\frac{1}{2})$. Let $\bar{\by}\in W_\infty$ and $\bp \in W_\infty(V_0^{1+\varepsilon},(V_0^{1-\varepsilon})').$ Then
\begin{align*}
 \left\| (\nabla \bar{\by})^T \bp \right\|_{L^2(0,\infty;(V_0^{1-\varepsilon})')} &\le M_1 \| \bar{\by} \| _{L^2(0,\infty;V)} \| \bp \| _{W_\infty(V_0^{1+\varepsilon},(V_0^{1-\varepsilon})')} \\
  \left\| (\bar{\by} \cdot \nabla ) \bp \right\|_{L^2(0,\infty;(V_0^{1-\varepsilon})')} &\le M_2 \| \bar{\by} \| _{L^\infty(0,\infty;Y)} \| \bp \| _{W_\infty(V_0^{1+\varepsilon},(V_0^{1-\varepsilon})')} .
 \end{align*}
\end{lemma}
\begin{proof}
For the first assertion, consider $\bp_i \frac{\partial \bar{\by}_j}{\partial x_k}$ with $i,j,k\in \{1,2\}$. Set $\lambda=\varepsilon -1, \mu = 1,\omega =1 -\varepsilon$. Then
\begin{align*}
 \mu +\omega + \lambda = 1 + (1-\varepsilon) + (\varepsilon - 1) = 1 \ge \frac{n}{2}, \\
 \mu >0 , \  \omega > 0, \ 2\lambda + \mu + \omega = 2(\varepsilon -1) + 1 +(1-\varepsilon) = \varepsilon > 0.
\end{align*}
Applying Proposition \ref{prop:grub} with $f=\bp_i$ and $g = \frac{\partial \bar{\by}_j}{\partial x_k}$ yields
\begin{align*}
\int_0^\infty \| fg \|_{H^{\varepsilon-1}(\Omega)}^2 \,\mathrm{d}t \le M \int_0^\infty \| \bp_i \| _{H^\varepsilon(\Omega)}^2 \left\| \frac{\partial \bar{\by}_j}{\partial x_k} \right\|_{L^2(\Omega)}^2 \, \mathrm{d}t
\end{align*}
which shows the first statement. For the second statement, set $f= \by_i$, $g= \frac{\partial \bp_j}{\partial x_k} , \lambda=\varepsilon-1$, $\mu=1-\varepsilon$ and $\omega=1$.
\end{proof}

The following lemma is formulated for an abstract generator $\widetilde{A}$ of an analytic exponentially stable semigroup on $Y$. It will subsequently be applied with $\widetilde{A}=A_\alpha$.

\begin{lemma}\label{lem:backwards}
Let $\widetilde{A} \in \mathcal{L}(V,V')$ generate an exponentially stable semigroup on $Y$ and assume that there exists $M\ge 0$ such that for every $\mathbf{f} \in L^2(0,\infty;V')$ there exists a unique $\by \in W_\infty$ satisfying
\begin{align*}
\dot{\by} &= \widetilde{A} \by + \mathbf{f} \ \ \text{on } [0,\infty) , \ \ \by(0) = 0 , \ \  \| \by \| _{W_\infty} \le M \| \mathbf{f} \| _{L^2(0,\infty;V')}.
\end{align*}
Then there exists $\tilde{M}$ such that for all $\Phi \in L^2(0,\infty;V')$ there exists a unique $\mathbf{r} \in W_\infty$ such that
\begin{align*}
-\dot{\mathbf{r}}  = \widetilde{A}^* \mathbf{r} + \Phi, \ \  \| \mathbf{r} \| _{W_\infty} \le \tilde{M} \| \Phi \| _{L^2(0,\infty;V')}.
\end{align*}
\end{lemma}
\begin{proof}
\emph{Step 1}.
Let us define $T\colon W^0_\infty \to L^2(0,\infty;V')$ by $T\by = \dot{\by} - \widetilde{A}\by$. Considering the adjoint $T^*\colon L^2(0,\infty;V) \to (W^0_\infty)'$ we have:
\begin{align*}
\langle T^* \boldsymbol{\varphi} , \by \rangle _{(W^0_\infty)',W^0_\infty} := \langle \boldsymbol{\varphi}, T\by \rangle_{L^2(0,\infty;V),L^2(0,\infty;V')} = \langle \boldsymbol{\varphi}, \dot{\by} - \widetilde{A} \by \rangle _{L^2(0,\infty;V),L^2(0,\infty;V')}.
\end{align*}
Since, by assumption, $T$ is a homeomorphism it is in particular surjective and injective and by the closed range theorem there exists a constant $C$ such that
\begin{align}\label{eq:back_aux1}
 \| \boldsymbol{\varphi} \| _{L^2(0,\infty;V)} \le C \| T^* \boldsymbol{\varphi} \| _{(W^0_\infty)'} , \ \ \forall \boldsymbol{\varphi} \in L^2(0,\infty;V).
\end{align}
\emph{Step 2}.
Let $\boldsymbol{\Phi}\in L^2(0,\infty;V')$ be arbitrary. Then there exists a unique $\mathbf{r} \in L^2(0,\infty;V)$ such that $T^* \mathbf{r} = \boldsymbol{\Phi}$, and by \eqref{eq:back_aux1} we have $\| \mathbf{r} \|_{L^2(0,\infty;V)} \le C \| \boldsymbol{\Phi} \|_{(W^0_\infty)'} \le C\| \boldsymbol{\Phi} \| _{L^2(0,\infty;V')}.$
Since $T^* \mathbf{r} = \boldsymbol{\Phi}$ we have for all $\by \in W^0_\infty$
\begin{align*}
 \langle \boldsymbol{\Phi}, \by \rangle _{L^2(0,\infty;V'), L^2(0,\infty;V)} &= \langle \mathbf{r},T\by \rangle _{L^2(0,\infty;V),L^2(0,\infty;V')}\\
 & = \langle \mathbf{r} , \dot{\by} \rangle _ {L^2(0,\infty;V),L^2(0,\infty;V')} - \langle \widetilde{A}^* \mathbf{r},\by \rangle _{L^2(0,\infty;V'),L^2(0,\infty;V)}.
\end{align*}
This implies that the time derivative of $\mathbf{r}$, in the sense of distributions, can be extended to a linear form on $W_\infty^0$ with the formula:
\begin{align}
\langle \dot{\mathbf{r}}, \by \rangle_{(W_\infty^0)',W_\infty^0}
= & \ - \langle \mathbf{r},\dot{\by} \rangle_{L^2(0,\infty;V),L^2(0,\infty;V')} \notag \\
= & \ - \langle \boldsymbol{\Phi} + \widetilde{A}^* \mathbf{r}, \by \rangle_{L^2(0,\infty;V'),L^2(0,\infty;V)}, \ \forall \mathbf{y}\in W^0_\infty. \label{eq:aux9}
\end{align}
We estimate
\begin{align*}
\langle \boldsymbol{\Phi} + \tilde{A}^* \mathbf{r}, \by \rangle_{L^2(0,\infty;V'),L^2(0,\infty;V)} &\le \| \boldsymbol{\Phi} +  \widetilde{A}^* \mathbf{r} \| _{L^2(0,\infty;V)} \| \mathbf{y} \|_{L^2(0,\infty;V)} \\
  &\le \|\boldsymbol{\Phi} \| _{L^2(0,\infty;V')} \| \mathbf{y} \|_{L^2(0,\infty;V)} + C \| \boldsymbol{\Phi} \|_{L^2(0,\infty;V')} \| \mathbf{y} \|_{L^2(0,\infty;V)} \\
  & = ( 1+ C) \| \boldsymbol{\Phi} \|_{L^2(0,\infty;V')} \| \mathbf{y} \|_{L^2(0,\infty;V)}.
\end{align*}
Together with \eqref{eq:aux9} and recalling that $W_\infty^0$ is dense in $L^2(0,\infty; V)$, we obtain that $\dot{\mathbf{r}}$ can be extended to a bounded linear form on $L^2(0,\infty;V)$, i.e., \@ $\dot{\mathbf{r}}$ can be extended to an element of $L^2(0,\infty;V')$, moreover,
\begin{equation*}
\| \dot{\mathbf{r}} \|_{L^2(0,\infty;V')} \leq ( 1+ C) \| \boldsymbol{\Phi} \|_{L^2(0,\infty;V')}.
\end{equation*}
It follows that $\mathbf{r} \in W_\infty$. Moreover,
\begin{align*}
 \| \mathbf{r} \| _{W_\infty} \le 2(1+C) \| \boldsymbol{\Phi} \|_{L^2(0,\infty;V')} \ \text{ and } -\dot{\mathbf{r}}- \widetilde{A}^* \mathbf{r} = \boldsymbol{\Phi} \ \ \text{in } L^2(0,\infty;V').
\end{align*}

\end{proof}

\begin{corollary}\label{cor:backwards_eps}
Let $\varepsilon \in (0,\frac{1}{2}).$ For all $\boldsymbol{\Phi} \in L^2(0,\infty;(V_0^{1-\varepsilon})')$, the system
\begin{equation*}
-\dot{\mathbf{r}}= A_{\alpha}^* \mathbf{r} + \boldsymbol{\Phi}
\end{equation*}
has a unique solution $\mathbf{r} \in W_\infty(V_0^{1+\varepsilon},(V_0^{1-\varepsilon})')$.
Moreover, there exists a constant $M_\alpha>0$ independent of $\boldsymbol{\Phi}$ such that
\begin{equation} \label{eq:estimate_adjoint_sys}
\| \mathbf{r} \|_{W_\infty(V_0^{1+\varepsilon},(V_0^{1-\varepsilon})')} \leq M_\alpha \| \boldsymbol{\Phi} \|_{L^2(0,\infty;(V_0^{1-\varepsilon})')}.
\end{equation}
\end{corollary}
\begin{proof}

 We appy Lemma \ref{lem:backwards} to $
 \dot{\bz} = A_{\alpha}^* \bz + \boldsymbol{\Psi}$ with
 $\boldsymbol{\Psi}=(-A_{\alpha}^*)^{\frac{\varepsilon}{2}} \boldsymbol{\Phi}\in L^2(0,T;V')$, to obtain
$ \| \mathbf{z} \| _{W_\infty} \le \tilde{M} \| \boldsymbol{\Phi} \| _{L^2(0,\infty;V')}$. Setting $\br:= (-A_{\alpha}^*)^{-\frac{\varepsilon}{2}} \mathbf{z}$
and using that $(-A_\alpha^*)^{-\frac{\varepsilon}{2}}$ is an isomorphism from $V$ to $V_0^{1+\varepsilon}$ and from $V'$ to $(V_0^{1-\varepsilon})'$, \cite[Section 1.15.2, p.101]{Tri78}, the claim follows.

\end{proof}



\begin{proof}[Proof of Proposition \ref{proposition:more_reg}]
 Only regularity has to be shown.
 Let us fix $\varepsilon \in (0,\frac{1}{2})$ and let $M_\alpha$ be given by Corollary \ref{cor:backwards_eps}.
  Let us define
 \begin{align*}
  \mathcal{M}&:=\left\{ \bp \in W_\infty(V_0^{1+\varepsilon},(V_0^{1-\varepsilon})') \, | \, \| \bp \|_{W_\infty(V_0^{1+\varepsilon},(V_0^{1-\varepsilon})')} \le 2 M_\alpha\gamma \right\} , \\
  %
\gamma&:=  \alpha\|\bp \|_{L^2(0,\infty;V)} + \| \bar{\by} \|_{L^2(0,\infty;V)} .
 \end{align*}
 Let us then choose $\tilde{\delta}_4>0$ such that
 \begin{align*}
 \| \bar{\by} \| _{L^2(0,\infty;V)} \le M \| \by_0 \| _Y  \le M\tilde{\delta}_4\le \frac{1}{2(M_1+M_2)M_\alpha},
 \end{align*}
 where $M_1$ and $M_2$ are given by Lemma \ref{lem:grubsol}.
 Further consider the mapping $\mathcal{Z}$ defined by
 \begin{align*}
  \mathcal{Z}& \colon \mathcal{M} \ni \mathbf{q} \mapsto \mathbf{r} \in  W_\infty(V_0^{1+\varepsilon},(V_0^{1-\varepsilon})') ,
 \end{align*}
 where $\mathbf{r}$ is the unique solution of
 \begin{align*}
  -\dot{\br} = A_\alpha^{*}\br + (\bar{\by}\cdot \nabla)\mathbf{q}-(\nabla \bar{\by})^T \mathbf{q} + \alpha \mathbf{p} + \bar{\by}
 \end{align*}
 according to Lemma \ref{lem:grubsol} and Corollary \ref{cor:backwards_eps}.
 Given $\mathbf{q} \in \mathcal{M}$, it holds that
 \begin{align*}
\| \br \| _{W_\infty(V_0^{1+\varepsilon},(V_0^{1-\varepsilon})')} &\le M_\alpha \left( \| (\bar{\by}\cdot \nabla)\mathbf{q}-(\nabla \bar{\by})^T \mathbf{q} + \alpha \mathbf{p} + \bar{\by} \| _{L^2(0,\infty;(V_0^{1-\varepsilon})')}\right) \\
 &\le M_\alpha \left(  (M_1+M_2){ \| \bar{\by}\|_{L^2(0,\infty;V)}} \| \mathbf{q} \|_{W_\infty(V_0^{1+\varepsilon},(V_0^{1-\varepsilon})')}+  \gamma \right) \\
   &\le M_\alpha \left( \frac{1}{2M_\alpha} \| \mathbf{q} \|_{W_\infty(V_0^{1+\varepsilon},(V_0^{1-\varepsilon})')}+  \gamma \right) \le 2M_\alpha \gamma .
 \end{align*}
 We obtain that $\mathcal{Z}(\mathcal{M})\subseteq \mathcal{M}$. Consider $\mathbf{q}_1,\mathbf{q}_2 \in \mathcal{Z}$ and let $\br=\mathcal{Z}(\mathbf{q}_1)-\mathcal{Z}(\mathbf{q}_2)$. Note that $\br$ solves
 \begin{align*}
- \dot{\br} = A_\alpha^*\mathbf{r} + (\bar{\by}\cdot \nabla)(\mathbf{q}_1-\mathbf{q}_2) - (\nabla \bar{\by})^T (\mathbf{q}_1-\mathbf{q}_2)
\end{align*}
so that we obtain
\begin{align*}
\|\mathcal{Z}(\mathbf{q}_1)-\mathcal{Z}(\mathbf{q}_2)\|_{W_\infty(V_0^{1+\varepsilon},(V_0^{1-\varepsilon})')} &= \| \br \| _{W_\infty(V_0^{1+\varepsilon},(V_0^{1-\varepsilon})')}\\
 &\le M_\alpha (M_1+M_2) \| \bar{\by}\|_{W_\infty} \| \mathbf{q}_1-\mathbf{q}_2\| _{W_\infty(V_0^{1+\varepsilon},(V_0^{1-\varepsilon})')} \\
&\le \frac{1}{2} \| \mathbf{q}_1-\mathbf{q}_2\| _{W_\infty(V_0^{1+\varepsilon},(V_0^{1-\varepsilon})')}.
\end{align*}
Thus by the Banach fixed point theorem we conclude that there exists $\mathbf{r} \in W_\infty(V_0^{1+\varepsilon},(V_0^{1-\varepsilon})') \subset W_\infty$ which is a solution of
\begin{align*}
 -\dot{\mathbf{r}} = A_\alpha^* \mathbf{r} + (\bar{\by}\cdot \nabla) \mathbf{r} - (\nabla \bar{\by})^T \mathbf{r} + \alpha \mathbf{p} + \mathbf{\bar{\by}}.
\end{align*}
It remains to show that $\mathbf{r}$ solves \eqref{eq:costate_NLQ}. For this, we define $\mathbf{e}:=\mathbf{r}-\mathbf{p} \in L^2(0,\infty;V)$ and observe that $\mathbf{e}$ satisfies
\begin{equation*}
T^*\mathbf{e}= G^*\mathbf{e} \quad \text{in $(W^0_\infty)'$},
\end{equation*}
where the operator $G \in \mathcal{L}(W_\infty,L^2(0,\infty;V'))$ is defined in \eqref{eq:G_aux}.
It follows from \eqref{eq:back_aux1} that
\begin{equation*}
\| \mathbf{e} \|_{L^2(0,\infty;V)}
\leq M \| T^* \mathbf{e} \|_{(W^0_\infty)'}
\leq M \| G^* \|_{\mathcal{L}(L^2(0,\infty;V),(W_\infty)')} \| \mathbf{e} \|_{L^2}(0,\infty;V).
\end{equation*}
As a consequence of \eqref{eq:estimate_g_aux}, $\tilde{\delta}_4$ can be reduced so that $\| G^* \| = \| G \| < \frac{1}{M}$. Hence, we obtain $\mathbf{e}=0$ and thus $\mathbf{r}=\mathbf{p}$ showing that $\mathbf{p} \in W_\infty(V_0^{1+\varepsilon},(V_0^{1-\varepsilon})') \subset W_\infty$.
\end{proof}


\end{document}